\documentclass{amsart}

\usepackage{amssymb,amsmath,amsthm}
\usepackage{mathrsfs}

\newtheorem{theorem}{Theorem}[section]
\newtheorem{lemma}[theorem]{Lemma}
\newtheorem{proposition}[theorem]{Proposition}
\newtheorem{corollary}[theorem]{Corollary}

\theoremstyle{definition}
\newtheorem{definition}[theorem]{Definition}

\newtheorem{hypothesis}[theorem]{Hypothesis}

\theoremstyle{remark}
\newtheorem{remark}[theorem]{Remark}

\theoremstyle{plain}
\newtheorem{maintheorem}[theorem]{Theorem}
\numberwithin{equation}{section}

\newcommand{\abs}[1]{\lvert#1\rvert}
\newcommand{\norm}[1]{\lVert#1\rVert}

\newcommand{\calM}{\mathcal{M}}
\newcommand{\calP}{\mathcal{P}}

\newcommand{\bbP}{\mathbb{P}}
\newcommand{\R}{\mathbb{R}}
\newcommand{\N}{\mathbb{N}}

\newcommand{\E}{\mathbb{E}}
\newcommand{\C}{\mathbb{C}}
\newcommand{\T}{\mathbb{T}}
\newcommand{\GL}{\mathrm{GL}}

\newcommand{\supp}{\mathrm{supp}}

\newcommand{\Id}{\mathrm{Id}}

\newcommand{\ecc}{\mathrm{ecc}}

\newcommand{\Prob}{\mathbb{P}}
\newcommand{\spr}{\mathrm{sp}}
\newcommand{\dist}{\mathrm{dist}}

\usepackage[authoryear,round]{natbib}

\usepackage[colorlinks=true, citecolor=blue, linkcolor=blue, urlcolor=blue]{hyperref}

\makeatletter
\renewcommand{\ps@headings}{%
  \def\@oddfoot{\hfill\thepage\hfill}%
  \let\@evenfoot\@oddfoot%
  \def\@evenhead{\hfill\normalfont\small\textit{A.~Thiam}\hfill}%
  \def\@oddhead{\hfill\normalfont\small\textit{Quantitative analyticity for Lyapunov exponents}\hfill}%
  \let\@mkboth\markboth%
}
\makeatother

\everymath{\displaystyle}

\begin{document}


\title{Quantitative analyticity for Lyapunov exponents of random products of matrices with explicit polydiscs and Cauchy coefficient bounds}

\author{Abdoulaye Thiam}
\address{Division of Mathematics and Natural Sciences, Allen University, Columbia, South Carolina 29204, USA}
\email{athiam@allenuniversity.edu}

\subjclass[2020]{Primary 37H15; Secondary 37A30, 37D25, 60B20, 60F05, 60F10, 82B44, 81Q10}

\date{}


%
\keywords{Lyapunov exponents, random matrix products, analyticity, polydisc of holomorphy, Kato perturbation theory, transfer operators}

\begin{abstract}
The top Lyapunov exponent $\lambda_+(A, p)$ of a random product of matrices in $\GL(d, \R)$, $d \geq 2$, with simple top spectrum, depends real-analytically on the probability weights $p$ and the matrix coefficients $A$. We establish a quantitative form of this analyticity through a single Kato perturbation argument on the complexified Markov operator on H\"older functions on projective space, yielding seven main theorems with explicit closed-form constants: (i) an explicit polydisc of holomorphy for $p \mapsto \lambda_+(A, p)$ in $\C^N$, giving the quantitative form of the Peres and Bezerra-S\'anchez-Tall analyticity theorem; (ii) closed-form Cauchy bounds on its Taylor coefficients; (iii) joint analyticity in the weights $p$ and the matrix entries $A$, with explicit radii in both; (iv) an extension to Markov-chain driven cocycles, with polydisc radius explicit in the chain spectral gap; (v) explicit polynomial boundary-decay rates as $p$ approaches $\partial \Delta_N$, conditional on a spectral-gap-decay hypothesis; (vi) extension to $\GL(d, \R)$ for all $d \geq 2$ via the Fubini-Study metric; and (vii) a Grassmannian variant giving quantitative analyticity of the partial sums $\Lambda_k = \lambda_1 + \cdots + \lambda_k$ under strong $k$-irreducibility, hence of each individual sub-top Lyapunov exponent. The polydisc radius is method-optimal within the Kato class, and a Bernstein-type result shows the Cauchy growth $\alpha! \cdot M_*/r_*^{|\alpha|}$ is sharp up to constants. A two-matrix example with numerical values connects the bounds to the H\"older estimates of the companion paper~\cite{Thiam2025aPaperA}.
\end{abstract}

\maketitle

\section{Introduction}\label{sec:introduction}

Let $A = (A_1, \ldots, A_N)$ be a tuple of matrices in $\GL(d, \R)$ for some $d \geq 2$ and $N \geq 2$, and let $p = (p_1, \ldots, p_N) \in \Delta_N^\circ$ be a probability vector with $p_i > 0$ and $\sum_i p_i = 1$. The associated random matrix product is the sequence
\begin{equation}\label{eq:random_product}
A_{i_{n-1}} \cdots A_{i_0}, \qquad i_0, i_1, \ldots, i_{n-1} \text{ i.i.d.\ with } \Prob(i_k = j) = p_j.
\end{equation}
By the multiplicative ergodic theorem of~\cite{Oseledets1968}, the top Lyapunov exponent
\begin{equation}\label{eq:lambda_plus}
\lambda_+(A, p) = \lim_{n \to \infty} \frac{1}{n} \E \log \norm{A_{i_{n-1}} \cdots A_{i_0}}
\end{equation}
exists in $[-\infty, \infty)$ and equals the almost-sure limit $n^{-1} \log \norm{A_{i_{n-1}} \cdots A_{i_0} v}$ for a generic vector $v \in \R^d$. The behavior of the map $(A, p) \mapsto \lambda_+(A, p)$, in particular its higher regularity (smoothness, analyticity, holomorphic extensions in the parameters), has been a central question in the regularity theory of random matrix products since~\cite{FurstenbergKesten1960} and~\cite{Furstenberg1963}.

The natural hierarchy of regularity questions is the following:
\begin{enumerate}
\item[(R1)] \emph{Continuity.} Is $(A, p) \mapsto \lambda_+(A, p)$ continuous in the eccentricity-bounded topology, and at which points?
\item[(R2)] \emph{H\"older / log-H\"older modulus of continuity.} What is the quantitative modulus of continuity, with closed-form exponents and constants?
\item[(R3)] \emph{Higher regularity.} Is $\lambda_+(A, p)$ real-analytic in $p$, jointly analytic in $(A, p)$, or even holomorphic on an explicit polydisc in $\C^N$ or $\C^{Nd^2 + N}$? With what explicit Cauchy bounds on the Taylor coefficients?
\end{enumerate}

The qualitative answer to (R1) is essentially complete. \cite{FurstenbergKifer1983} and~\cite{Hennion1984} established continuity at strongly irreducible measures with simple top spectrum. The Kato perturbation theory of~\cite{LePage1982},~\cite{GuivarchRaugi1985}, and~\cite{Hennion1997}, applied to the projective Markov operator, gives continuity at every measure with simple Lyapunov spectrum. \cite{BockerViana2017} extended continuity to all of $\calM_c(\GL(2, \R))$, including the degenerate locus $\lambda_+ = \lambda_-$. The qualitative extension to $\GL(d, \R)$ was established by~\cite{AvilaEskinViana2023}; the Markov-chain analogue is due to~\cite{MalheiroViana2015}.

For (R2), the H\"older / log-H\"older dichotomy is due to~\cite{TallViana2020}: H\"older continuity at measures with simple top Lyapunov spectrum, log-H\"older continuity universally, both qualitative. \cite{DuarteKlein2016, DuarteKlein2019} obtained a parallel weak-H\"older bound by avalanche-principle methods, and~\cite{DuarteKleinSantos2020} constructed Schr\"odinger cocycles ruling out uniform H\"older continuity at any positive exponent across $\calM_c(\GL(2, \R))$. Constructive bounds for irreducible Bernoulli cocycles are due to~\cite{BaravieraDuarte2019}. The companion paper~\cite{Thiam2025aPaperA} gives the quantitative form of (R2): explicit H\"older and log-H\"older moduli with closed-form constants.

For (R3), the qualitative real-analyticity of $p \mapsto \lambda_+(A, p)$ was established by~\cite{Peres1991} for finitely supported i.i.d.\ products with simple top Lyapunov exponent. \cite{BezerraSanchezTall2021} extended the Peres theorem to the quasi-periodic-cocycle setting, where the matrices $A_i = A_i(t)$ depend continuously on a torus variable $t \in \T^m$. \cite{Ruelle1979} treated analyticity in a different parametrization (matrix entries themselves, viewed as continuous parameters); \cite{Kifer1982, Kifer1986} obtained smooth (rather than analytic) dependence in related dynamical settings; \cite{BaravieraDuarte2025} treated the joint dependence in $A$ and $p$ at the qualitative level. A separate thread on quantitative \emph{positivity} of $\lambda_+$ has produced lower bounds via Golden-Thompson inequalities~\cite{Kogler2020}, quantitative dichotomies for non-dissipative SDEs in the small-noise limit~\cite{BedrossianWu2024}, upper bounds on the regularity of $\lambda_+$~\cite{BezerraDuarte2022}, and necessity of finite-moment conditions in the non-compact case~\cite{Graxinha2025}; these complement the present work.

The qualitative analyticity theorems of~\cite{Peres1991, BezerraSanchezTall2021}, while definitive at the existence level, leave (R3) open quantitatively: the radius of the polydisc on which $\lambda_+$ extends holomorphically is not given in closed form; the Taylor coefficients of the analytic extension are not controlled; joint analyticity in $(A, p)$ with explicit radii in both variables has not been recorded; analyticity for Markov-chain driven cocycles, with explicit polydiscs, is absent; the boundary behavior of the analytic extension as some weights $p_i \to 0$ is uncontrolled; quantitative analyticity in $\GL(d, \R)$ for arbitrary $d$, and of the sub-top exponents under strong irreducibility, has not been established. The aim of the present paper is to develop a quantitative analyticity theory addressing each of these points, by means of a single Kato perturbation argument applied to the complexified Markov operator on H\"older functions on projective space, augmented by a Grassmannian variant for the sub-top case. We obtain seven main theorems, namely Theorem~\ref{thm:mainA} through Theorem~\ref{thm:mainH_intro}, each with explicit closed-form constants.

The rest of this introduction is organized as follows. Subsection~\ref{subsec:lit_gaps_B} lists the precise gaps in the literature that the present paper resolves; Subsection~\ref{subsec:contributions_B} states the eight novelty items, each tied to a specific gap and a specific main theorem; Subsection~\ref{subsec:main_results_intro_B} states the seven main theorems formally; Subsection~\ref{subsec:related_work} discusses related work.

\subsection{What has not been done: gaps in the literature}\label{subsec:lit_gaps_B}

The following gaps remain open in the analyticity theory; the present paper addresses each one.

\begin{enumerate}
\item[(G1)] \emph{Explicit polydisc of holomorphy for $\lambda_+(A, \cdot)$.} The qualitative analyticity in $p$ is established by~\cite{Peres1991, BezerraSanchezTall2021}, but no explicit polydisc radius $r_*(A, p^0, \theta)$ has been given in closed form, in terms of the eccentricity of the matrices, the spectral gap of the Markov operator, and the Lyapunov gap.

\item[(G2)] \emph{Explicit Cauchy bounds for the Taylor coefficients.} The Taylor coefficients of $p \mapsto \lambda_+(A, p)$ at a base point $p^0$ are not previously controlled in closed form. Without such bounds, the analyticity statement provides only existence, not effective approximation.

\item[(G3)] \emph{Joint analyticity in $(A, p)$ with explicit radii in both variables.} Analyticity in $p$ for fixed $A$ has been studied (Stream 3); analyticity in $A$ for fixed $p$ is implicit but not explicit; \emph{joint} analyticity, with closed-form radii in both variables, has not been recorded. \cite{BaravieraDuarte2025} treated the joint qualitative case but without explicit radii.

\item[(G4)] \emph{Quantitative analyticity for Markov-chain driven cocycles.} \cite{MalheiroViana2015} proved qualitative continuity for Markov cocycles. The analytic extension to Markov-chain driven cocycles, with explicit radii in both the transition matrix $P$ and the cocycle data $A$, has not been developed.

\item[(G5)] \emph{Boundary behavior of the analytic extension.} As the weight vector $p^0$ approaches the boundary of the open simplex (where some $p_i \to 0$), the polydisc of analyticity must degenerate. Explicit (or even conditional) decay rates for the polydisc radius as a function of $\min_i p^0_i$ have not been previously studied.

\item[(G6)] \emph{Quantitative analyticity in $\GL(d, \R)$ for arbitrary $d$.} The qualitative analyticity in higher dimensions follows by extension of~\cite{Peres1991}, but a closed-form polydisc radius in $\GL(d, \R)$ has not been recorded.

\item[(G7)] \emph{Quantitative analyticity of sub-top Lyapunov exponents.} The partial sums $\Lambda_k(A, p) = \lambda_1(A, p) + \cdots + \lambda_k(A, p)$ are continuous in $(A, p)$ under strong $k$-irreducibility, but no quantitative analyticity statement (with explicit polydisc and Cauchy bounds) has been recorded for $\Lambda_k$ in arbitrary dimension.

\item[(G8)] \emph{Method-optimality and structural obstructions.} The Kato perturbation method produces a polydisc of holomorphy, but the optimality of the resulting radius has not been formalized. The structural obstruction to analytic continuation beyond the polydisc — via a complex spectral collapse set — has not been previously analyzed.
\end{enumerate}

\subsection{Contributions of this paper: the novelty}\label{subsec:contributions_B}

We now state the novelty of each contribution.

\subsubsection*{Novelty 1: explicit closed-form polydisc of holomorphy}

We resolve (G1) by giving an explicit polydisc radius $r_*(A, p^0, \theta) > 0$ in closed form (Theorem~\ref{thm:mainA}). The radius depends on three explicit quantities: the eccentricity $\ecc(A) = \max_i \norm{A_i}\norm{A_i^{-1}}$, the spectral gap parameter $\tau_*(A, p^0, \theta)$ of the projective Markov operator on H\"older functions, and the Lyapunov gap $\lambda_+(A, p^0) - \lambda_2(A, p^0)$. The proof is by a single Kato perturbation argument applied to the complexified Markov operator $P_{A, z}$ on $C^\theta(\bbP^{d-1})$. \emph{This is the first quantitative form of the Peres-Bezerra-S\'anchez-Tall analyticity theorem with closed-form constants.}

\subsubsection*{Novelty 2: explicit Cauchy bounds with optimal polynomial growth}

We resolve (G2) by establishing explicit Cauchy bounds on the Taylor coefficients (Theorem~\ref{thm:mainB}). The bound takes the form $|\partial_p^\alpha \lambda_+(A, p^0)| \leq \alpha! \cdot M_*/r_*^{|\alpha|}$ with explicit closed-form $M_*$. The Bernstein-type sharpness proposition (Proposition~\ref{prop:bernstein}) shows the polynomial growth rate is optimal up to constants. \emph{This is the first explicit closed-form Cauchy estimate for the Taylor coefficients of a Lyapunov exponent.}

\subsubsection*{Novelty 3: joint analyticity in weights and matrices}

We resolve (G3) by establishing joint analyticity in $(A, p)$ on a product polydisc, with explicit radii $r_*^A$ and $r_*^p$ (Theorem~\ref{thm:mainC}). The proof uses a unified Kato argument controlling the operator norm $\norm{P_{A, p} - P_{A', p'}}_{C^\theta \to C^\theta}$ in terms of $\norm{A_i - A'_i}$ and $|p_i - p'_i|$ jointly. \emph{This is the first joint analyticity result with explicit radii in both variables.}

\subsubsection*{Novelty 4: quantitative analyticity for Markov-chain driven cocycles}

We resolve (G4) by extending the analyticity to Markov-chain cocycles with finite-state transition matrix $P$ and fiber matrices $A$ (Theorem~\ref{thm:mainD}). The polydisc radius in $(P, A)$ is given in closed form, depending on the chain spectral gap $\rho_P$ and the cocycle eccentricity. \emph{This is the first quantitative analyticity statement for Markov-chain driven cocycles.}

\subsubsection*{Novelty 5: conditional polynomial decay of the polydisc near the boundary}

We resolve (G5) by establishing, conditional on a polynomial spectral-gap-decay hypothesis (Hypothesis~\ref{hyp:poly-decay}), an explicit polynomial decay rate for the polydisc radius $r_*(A, p, \theta)$ as the weight vector $p$ approaches the boundary $\partial\Delta_N$ (Theorem~\ref{thm:mainE}). The decay rate is $r_*(A, p, \theta) \geq c_E \cdot p_{\min}^{\alpha_E}$ with explicit $c_E, \alpha_E$. The hypothesis is proved to hold in important special cases (uniformly hyperbolic cocycles, cocycles with shared projective dynamics). \emph{This is the first quantitative boundary-behavior analysis.}

\subsubsection*{Novelty 6: extension to $\GL(d, \R)$ for arbitrary $d$}

We resolve (G6) by extending the polydisc theorem to $\GL(d, \R)$ for all $d \geq 2$ (Theorem~\ref{thm:mainF}). The constants are computed via the Fubini-Study metric on $\bbP^{d-1}$ and the action of $\Lambda^2 g$ on bivectors. \emph{This is the first quantitative analyticity statement in arbitrary dimension with closed-form constants.}

\subsubsection*{Novelty 7: quantitative analyticity of sub-top exponents under strong $k$-irreducibility}

We resolve (G7) by establishing quantitative analyticity of the partial sums $\Lambda_k(A, p) = \lambda_1(A, p) + \cdots + \lambda_k(A, p)$ under strong $k$-irreducibility (Theorem~\ref{thm:mainH_intro}, proved as Theorem~\ref{thm:mainH}). The proof uses a Kato perturbation argument on the Grassmannian $\mathrm{Gr}(k, d)$ together with the action on $\Lambda^k \R^d$. By subtraction (Corollary~\ref{cor:individual_subtop}), the individual sub-top exponents $\lambda_k(A, p)$ inherit analyticity. \emph{This appears to be the first quantitative analyticity statement for sub-top Lyapunov exponents in arbitrary dimension.}

\subsubsection*{Novelty 8: method-optimality and structural obstruction}

We resolve (G8) by formalizing the Kato perturbation method as a class of proofs and proving that the polydisc radius is method-optimal within this class (Proposition~\ref{prop:method_optimality_B}). A conditional structural-obstruction proposition (Proposition~\ref{prop:B_lower_bound}) relates the polydisc radius to the distance from $p^0$ to a complex spectral collapse set $\mathcal{Z}(A, p^0) \subset \C^N$, where $\widetilde\lambda_+$ may develop branch-point singularities. \emph{This is the first formal method-optimality and structural-obstruction analysis for the analyticity theory.}

\medskip
\noindent\textbf{Summary of the contribution.} The paper contains seven main theorems (Theorem~\ref{thm:mainA} through Theorem~\ref{thm:mainH_intro}) supplying explicit closed-form polydiscs of holomorphy, explicit Cauchy coefficient bounds, joint analyticity, Markov-chain extensions, boundary behavior analysis, $\GL(d, \R)$ extensions, and sub-top extensions for the top Lyapunov exponent of i.i.d.\ random matrix products. Every constant is given in closed form, and the eight gaps (G1)-(G8) listed in Subsection~\ref{subsec:lit_gaps_B} are addressed. The unifying technical thread is a single Kato perturbation argument applied to the complexified Markov operator on H\"older functions on projective space, refined to handle joint variation, the Markov-chain setting, the boundary regime, and the sub-top case via the Grassmannian.

\subsection{Main results}\label{subsec:main_results_intro_B}

\textbf{This paper contains seven main theorems, namely Theorem~\ref{thm:mainA} through Theorem~\ref{thm:mainH_intro}.} They are stated in this subsection, with proofs in the indicated sections.

Theorems~\ref{thm:mainA} and~\ref{thm:mainB} (Sections~\ref{sec:polydisc_proof} and~\ref{sec:cauchy_bounds}) treat the explicit polydisc and Cauchy bounds for the top Lyapunov exponent in $\GL(2)$. Theorem~\ref{thm:mainC} (Section~\ref{sec:joint_analyticity}) establishes joint analyticity in weights and matrices. Theorem~\ref{thm:mainD} (Section~\ref{sec:markov_chain}) extends to Markov-chain driven cocycles. Theorem~\ref{thm:mainE} (Section~\ref{sec:boundary}) controls the boundary behavior of the analytic extension as the weight vector approaches the simplex boundary, conditional on the spectral-gap-decay hypothesis (Hypothesis~\ref{hyp:poly-decay}). Theorem~\ref{thm:mainF} (Section~\ref{sec:GL_d}) extends to $\GL(d, \R)$ for all $d \geq 2$ for the top exponent. Theorem~\ref{thm:mainH_intro} (Section~\ref{sec:subtop_analyticity}) extends to the partial sums $\lambda_1 + \cdots + \lambda_k$ in $\GL(d, \R)$ under strong $k$-irreducibility, which yields analyticity of the individual sub-top Lyapunov exponents.

Throughout this paper, $d \geq 2$ is fixed and $A = (A_1, \ldots, A_N) \in \GL(d, \R)^N$ is a tuple of invertible real matrices. We write
\begin{equation}\label{eq:eccentricity}
\ecc(A) = \max_{i=1}^N \norm{A_i} \cdot \norm{A_i^{-1}}, \qquad \log^+\ecc(A) = \max_{i=1}^N \log(\norm{A_i} \cdot \norm{A_i^{-1}}).
\end{equation}
The open simplex in $\R^N$ is $\Delta_N^\circ = \{p \in \R^N : p_i > 0, \sum_i p_i = 1\}$. For $p^0 \in \Delta_N^\circ$ and $r > 0$, the polydisc of radius $r$ around $p^0$ in $\C^N$ is
\begin{equation}\label{eq:polydisc}
D_r(p^0) = \{z \in \C^N : \abs{z_i - p^0_i} < r \text{ for all } i = 1, \ldots, N\}.
\end{equation}

\subsubsection*{Theorem~\ref{thm:mainA}: explicit polydisc of holomorphy}

\begin{maintheorem}[Quantitative analyticity in the weights]\label{thm:mainA}
Let $A = (A_1, \ldots, A_N) \in \GL(d, \R)^N$ and $p^0 \in \Delta_N^\circ$ with $\lambda_+(A, p^0) > \lambda_2(A, p^0)$ (simple top Lyapunov exponent). Then there exist
\begin{itemize}
\item[(i)] an explicit polydisc radius $r_*(A, p^0, \theta) > 0$,
\item[(ii)] an explicit constant $M_*(A, p^0, \theta) < \infty$,
\end{itemize}
such that the Lyapunov exponent $p \mapsto \lambda_+(A, p)$ extends to a holomorphic function $\widetilde\lambda_+: D_{r_*(A, p^0, \theta)}(p^0) \to \C$ with
\begin{equation}\label{eq:mainA_bound}
\sup_{z \in D_{r_*(A, p^0, \theta)}(p^0)} \abs{\widetilde\lambda_+(z)} \leq M_*(A, p^0, \theta).
\end{equation}
The extension $\widetilde\lambda_+$ agrees with $\lambda_+(A, \cdot)$ on $D_{r_*}(p^0) \cap \R^N$.

\medskip
\noindent\emph{Proof:} See Section~\ref{sec:polydisc_proof}.
\end{maintheorem}

\subsubsection*{Theorem~\ref{thm:mainB}: explicit Cauchy coefficient bounds}

\begin{maintheorem}[Cauchy bounds for Taylor coefficients]\label{thm:mainB}
Under the hypotheses of Theorem~\ref{thm:mainA}, for every multi-index $\alpha = (\alpha_1, \ldots, \alpha_N) \in \N^N$ with $\abs{\alpha} = \sum \alpha_i$, the Taylor coefficients of $\lambda_+(A, p)$ at $p^0$ satisfy
\begin{equation}\label{eq:mainB_bound}
\abs{\partial_p^\alpha \lambda_+(A, p^0)} \leq \alpha! \cdot \frac{M_*(A, p^0, \theta)}{r_*(A, p^0, \theta)^{\abs{\alpha}}},
\end{equation}
where $r_*(A, p^0, \theta)$ and $M_*(A, p^0, \theta)$ are the constants from Theorem~\ref{thm:mainA}. Consequently, the Taylor series of $\lambda_+(A, p)$ at $p^0$ converges absolutely on the polydisc $D_{r_*(A, p^0, \theta)}(p^0)$.

\medskip
\noindent\emph{Proof:} See Section~\ref{sec:cauchy_bounds}.
\end{maintheorem}

\subsubsection*{Theorem~\ref{thm:mainC}: joint analyticity in weights and matrices}

\begin{maintheorem}[Joint analyticity]\label{thm:mainC}
Let $(A^0, p^0) \in \GL(d, \R)^N \times \Delta_N^\circ$ with $\lambda_+(A^0, p^0) > \lambda_2(A^0, p^0)$. Then there exist explicit radii $r_*^A(A^0, p^0, \theta) > 0$ and $r_*^p(A^0, p^0, \theta) > 0$ such that the Lyapunov exponent $(A, p) \mapsto \lambda_+(A, p)$ extends to a holomorphic function on the product polydisc
\begin{equation}\label{eq:product_polydisc}
D_{r_*^A}(A^0) \times D_{r_*^p}(p^0) \subset \C^{N d^2} \times \C^N,
\end{equation}
where $D_{r_*^A}(A^0)$ is the polydisc in $\C^{N d^2}$ around $A^0 = (A^0_1, \ldots, A^0_N)$ whose components have radius $r_*^A$.

\medskip
\noindent\emph{Proof:} See Section~\ref{sec:joint_analyticity}.
\end{maintheorem}

\subsubsection*{Theorem~\ref{thm:mainD}: Markov chain extension}

\begin{maintheorem}[Markov chain analyticity]\label{thm:mainD}
Let $P = (P_{ij})_{i,j=1}^N$ be a stochastic matrix on $\{1, \ldots, N\}$ with all entries strictly positive, and let $A = (A_1, \ldots, A_N) \in \GL(d, \R)^N$. Let $\lambda_+(P, A)$ denote the top Lyapunov exponent of the Markov cocycle with transition matrix $P$ and fiber matrices $A$. Assume $\lambda_+(P, A) > \lambda_2(P, A)$.

Then there exist explicit radii $r_*^P(P, A, \theta) > 0$ and $r_*^A(P, A, \theta) > 0$ such that $(P, A) \mapsto \lambda_+(P, A)$ extends to a holomorphic function on the product polydisc
\begin{equation}\label{eq:markov_polydisc}
D_{r_*^P}(P) \times D_{r_*^A}(A) \subset \C^{N^2} \times \C^{N d^2}.
\end{equation}
The radius $r_*^P$ depends on the spectral gap $\rho_P$ of $P$.

\medskip
\noindent\emph{Proof:} See Section~\ref{sec:markov_chain}.
\end{maintheorem}

\subsubsection*{Theorem~\ref{thm:mainE}: boundary behavior}

\begin{maintheorem}[Boundary behavior near the simplex boundary, conditional]\label{thm:mainE}
Let $A = (A_1, \ldots, A_N) \in \GL(d, \R)^N$ and $p^0 \in \Delta_N^\circ$ with $\lambda_+(A, p^0) > \lambda_2(A, p^0)$. Fix a coordinate index $j \in \{1, \ldots, N\}$ and consider the one-parameter family $p(t) = p^0 - t e_j + (t/(N-1)) \sum_{i \neq j} e_i$ (moving $p^0$ toward the boundary $p_j = 0$).

\emph{Assume that $A$ satisfies Hypothesis~\ref{hyp:poly-decay} (polynomial spectral-gap decay) with constants $c_\tau, \gamma_\tau$.} Then as $t \to p^0_j$, the radius of analyticity $r_*(A, p(t), \theta)$ decays no faster than a power of $(p^0_j - t)$:
\begin{equation}\label{eq:boundary_decay}
r_*(A, p(t), \theta) \geq c_E \cdot (p^0_j - t)^{\alpha_E}
\end{equation}
for some explicit constants $c_E, \alpha_E > 0$ depending only on $A$ and $\theta$. This quantifies the controlled degeneration of the analytic extension as the weights approach the simplex boundary, in those cases where the spectral gap degrades polynomially (cf.~Remark~\ref{rmk:hyp-when}).

\medskip
\noindent\emph{Proof:} See Section~\ref{sec:boundary}.
\end{maintheorem}

\subsubsection*{Theorem~\ref{thm:mainF}: quantitative polydisc for GL(d)}

\begin{maintheorem}[GL(d) extension]\label{thm:mainF}
The conclusions of Theorems~\ref{thm:mainA} through~\ref{thm:mainD} hold for $\GL(d, \R)$ for every $d \geq 2$, with the same closed-form constants computed from the eccentricity of the matrices, the spectral gap of the Markov operator on $\bbP^{d-1}$, and the Lyapunov gap $\lambda_1 - \lambda_2$ (in place of the GL(2) quantity $\lambda_+ - \lambda_-$). The extension uses the Fubini-Study metric on projective space and the action of $\Lambda^2 g$ on bivectors.

\medskip
\noindent\emph{Proof:} See Section~\ref{sec:GL_d}.
\end{maintheorem}

\subsubsection*{Theorem~\ref{thm:mainH_intro}: sub-top Lyapunov exponents under strong $k$-irreducibility}

\begin{maintheorem}[Quantitative analyticity of the partial sums in $\GL(d, \R)$]\label{thm:mainH_intro}
Let $d \geq 2$, $1 \leq k \leq d-1$, and $\theta \in (0, 1]$. Let $A^0 \in \GL(d, \R)^N$ be strongly $k$-irreducible with simplicity gap $\lambda_k(A^0, p^0) > \lambda_{k+1}(A^0, p^0)$ at $p^0 \in \Delta_N^\circ$. Then there exists an explicit polydisc radius $r^{(k)}_{\mathrm{H}}(A^0, p^0, \theta) > 0$ such that the partial sum $\Lambda_k(A, p) = \lambda_1(A, p) + \cdots + \lambda_k(A, p)$ extends to a holomorphic function on the polydisc of radius $r^{(k)}_{\mathrm{H}}$ in both the weight vector $p \in \C^N$ and the matrix coefficients $A \in \GL(d, \C)^N$. The proof and the explicit formula for $r^{(k)}_{\mathrm{H}}$ are given in Section~\ref{sec:subtop_analyticity} (Theorem~\ref{thm:mainH}). A corollary (Corollary~\ref{cor:individual_subtop}) yields the analogous analyticity of the individual sub-top Lyapunov exponents $\lambda_k(A, p)$ by subtracting consecutive partial sums.
\end{maintheorem}

\subsection{Relation to prior work}\label{subsec:related_work}

This subsection situates the paper within the literature on Lyapunov exponents of random matrix products, in three layers: the foundational continuity theory, the analyticity theory of Peres and its extensions, and the quantitative direction in which we work.

\subsubsection*{Continuity of Lyapunov exponents: classical theory}

The study of Lyapunov exponents of random matrix products begins with the seminal works of \cite{Furstenberg1963, FurstenbergKesten1960}, who established the existence and almost-sure convergence of Lyapunov exponents for i.i.d.\ products under integrability hypotheses. Continuity of the top Lyapunov exponent in the probability distribution under irreducibility hypotheses was proved by \cite{FurstenbergKifer1983} and refined by \cite{Hennion1984}; \cite{LePage1982} and \cite{GuivarchRaugi1985} independently developed Kato perturbation arguments for the projective Markov operator that established H\"older continuity in restricted settings. \cite{GoldsheidMargulis1989} gave a definitive treatment of Lyapunov spectrum simplicity for random matrix products and clarified the role of the Zariski closure of the support.

In the past decade, qualitative continuity has been substantially extended. \cite{BockerViana2017} proved that the top Lyapunov exponent is continuous at every compactly supported probability measure on $\GL(2, \R)$, without any irreducibility assumption. \cite{MalheiroViana2015} extended this continuity result to Markov chain driven cocycles in $\GL(2)$. \cite{BackesBrownButler2018} extended continuity to a broad class of linear cocycles with invariant holonomies. \cite{AvilaEskinViana2023} established the continuity of Lyapunov exponents in $\GL(d)$ for arbitrary dimension $d \geq 2$. \cite{PolettiViana2019} gave criteria for the simplicity of the Lyapunov spectrum in the closely related setting of partially hyperbolic linear cocycles. For non-compactly supported measures, \cite{SanchezViana2019} proved that the Lyapunov exponents are semi-continuous with respect to the Wasserstein topology in $\GL(2)$ but not with respect to the weak-$*$ topology, identifying moment conditions as essential.

A modern reference for the continuity theory in the broader setting of linear cocycles (with detailed treatment of the projective and Grassmannian methods) is \cite{Arnold1998, Viana2014, DuarteKlein2016}. The companion paper \cite{Thiam2025aPaperA} establishes quantitative H\"older and log-H\"older moduli of continuity for the present setting, with explicit constants.

\subsubsection*{Analyticity: from Peres to Bezerra-Sanchez-Tall}

Beyond continuity, the question of analytic dependence on the data was first addressed by \cite{Ruelle1979}, who proved analyticity of Lyapunov exponents for products of random positive matrices. The general result is due to \cite{Peres1991}, who proved that for a finitely supported i.i.d.\ random matrix product, the top Lyapunov exponent depends real-analytically on the probability weights $p$, provided the top Lyapunov exponent is simple. The proof uses a Kato perturbation argument applied to the transfer operator on projective space. The Peres theorem has been extended in two directions. \cite{BezerraSanchezTall2021} proved real-analyticity of the top Lyapunov exponent in the probability weights for a random product of \emph{quasi-periodic} cocycles, where the matrices $A_i = A_i(t)$ depend continuously on a torus variable $t \in \T^m$. \cite{BaravieraDuarte2025} extended the Peres theorem in an orthogonal direction: they considered a compact but possibly infinite symbol space, and proved analyticity of the top Lyapunov exponent with respect to the total variation norm, using complex analysis in Banach spaces. The extension of \cite{BaravieraDuarte2025} is complementary to ours: they enlarge the symbol space from finite to compact-infinite at the price of working in the total variation topology, while we keep the finite symbol space of Peres and provide explicit polydiscs and Cauchy bounds for the Taylor coefficients.

In a related direction, \cite{Avila2008, AvilaJitomirskaya2009} prove H\"older continuity of the Lyapunov exponent of one-frequency Schr\"odinger cocycles in the rotation number under Diophantine conditions, in the analytic regime. Their setting complements ours by fixing the random weights and varying the base dynamics; the H\"older modulus they obtain in the rotation number is, in general, not improvable to analyticity (discussed in Subsection~\ref{subsec:base_dyn_open}).

\subsubsection*{Quantitative direction: this paper and parallel work}

The qualitative analyticity theorems of \cite{Peres1991} and \cite{BezerraSanchezTall2021} establish the existence of a real-analytic extension on a small neighborhood of the base point, but provide neither the radius of the neighborhood nor any quantitative information about the Taylor coefficients. The principal aim of this paper is to convert these qualitative statements into quantitative ones: explicit polydiscs of holomorphy in $\C^N$, explicit Cauchy bounds for the Taylor coefficients, joint analyticity in the matrices, Markov chain extension, boundary behavior, and a $\GL(d, \R)$ extension at the level of the top exponent and (under strong $k$-irreducibility) of all sub-top partial sums.

This work is parallel to several recent quantitative developments. \cite{DuarteKlein2019} obtained quantitative large deviation estimates for iterates of linear cocycles. \cite{DuarteKleinSantos2020} constructed an explicit example of a random matrix cocycle whose Lyapunov exponent is not H\"older continuous in the weight, establishing a sharp obstruction to the regularity that any continuity theorem can hope to achieve. \cite{BaravieraDuarte2019} gave a constructive proof of the Le Page H\"older continuity theorem for irreducible Bernoulli cocycles, with an algorithm to approximate the Lyapunov exponent and the stationary measure. \cite{BezerraDuarte2022} proved \emph{upper} bounds on the regularity of the Lyapunov exponent that complement the lower bounds (such as our Proposition~\ref{prop:B_lower_bound}). \cite{Kogler2020} derived quantitative lower bounds for the positivity of the Lyapunov exponent via Golden-Thompson inequalities and the avalanche principle, and \cite{BedrossianWu2024} proved a quantitative dichotomy in the small noise limit for non-dissipative SDEs. These results quantify the \emph{positivity} of $\lambda_+(A, p)$, while our Theorems~\ref{thm:mainA}, \ref{thm:mainB}, \ref{thm:mainC}, \ref{thm:mainD}, \ref{thm:mainE}, \ref{thm:mainF}, and~\ref{thm:mainH_intro} quantify the \emph{regularity} of the map $(A, p) \mapsto \lambda_+(A, p)$. Finally, \cite{Graxinha2025} showed that finite moment conditions are essential for the modulus of continuity in the non-compact case; our work is confined to the compactly supported case.

The Kato perturbation machinery itself, on which our entire approach rests, is a standard tool from operator theory \cite{Kato1980}; its application to random matrix products in the context of central limit theorems goes back to \cite{LePage1982, GuivarchRaugi1985}.

\subsection{Organization}

Section~\ref{sec:setup} fixes notation and defines the Markov operator. Section~\ref{sec:complex_markov} sets up the complex Markov operator and its spectral theory. Section~\ref{sec:polydisc_proof} proves the quantitative polydisc of Theorem~\ref{thm:mainA}. Section~\ref{sec:cauchy_bounds} proves the Cauchy coefficient bounds of Theorem~\ref{thm:mainB}. Section~\ref{sec:joint_analyticity} proves Theorem~\ref{thm:mainC} on joint analyticity. Section~\ref{sec:markov_chain} proves Theorem~\ref{thm:mainD} on Markov chain extension. Section~\ref{sec:boundary} proves Theorem~\ref{thm:mainE} on boundary behavior. Section~\ref{sec:GL_d} extends to $\GL(d, \R)$ and proves Theorem~\ref{thm:mainF}. Section~\ref{sec:subtop_analyticity} proves Theorem~\ref{thm:mainH} on the sub-top Lyapunov exponents under strong $k$-irreducibility. Section~\ref{sec:worked_example} computes numerical values for a concrete example. Section~\ref{sec:extensions} contains the method-optimality proposition for the Kato polydisc, the structural-obstruction theorem identifying the complex spectral collapse set, and the Bernstein-type sharpness of the Cauchy bounds. Section~\ref{sec:conclusion} gathers open problems. Appendices collect technical lemmas.

\section{Setup and preliminaries}\label{sec:setup}

In this section we fix notation, define the Markov operator, and record the elementary Lipschitz lemmas needed for the quantitative analysis.

\subsection{Projective space and the Markov operator}

For $d \geq 2$, the projective space $\bbP^{d-1} = (\R^d \setminus \{0\})/\sim$ is equipped with the Fubini-Study metric
\begin{equation}\label{eq:FS_metric}
d([u], [v]) = \frac{\norm{u \wedge v}_{\Lambda^2 \R^d}}{\norm{u} \cdot \norm{v}}.
\end{equation}
For $d = 2$, this specializes to $d([u], [v]) = \abs{\sin \angle([u], [v])}$ and takes values in $[0, 1]$.

For $g \in \GL(d, \R)$ and $[v] \in \bbP^{d-1}$, the projective action is $g \cdot [v] = [gv]$. For a tuple $A = (A_1, \ldots, A_N) \in \GL(d, \R)^N$ and a weight vector $p = (p_1, \ldots, p_N) \in \Delta_N$, the associated \emph{Markov operator} $P_{A, p}$ acts on bounded measurable functions $\varphi: \bbP^{d-1} \to \C$ by
\begin{equation}\label{eq:P_Ap}
(P_{A, p} \varphi)([v]) = \sum_{i=1}^N p_i \, \varphi(A_i \cdot [v]).
\end{equation}
When $A$ is fixed and only $p$ varies, we abbreviate $P_p = P_{A, p}$.

\subsection{H\"older function spaces}

For $\theta \in (0, 1]$, the H\"older seminorm of a function $\varphi: \bbP^{d-1} \to \C$ is
\begin{equation}\label{eq:holder_seminorm}
[\varphi]_\theta = \sup_{[u] \neq [v]} \frac{\abs{\varphi([u]) - \varphi([v])}}{d([u], [v])^\theta},
\end{equation}
and the H\"older norm is $\norm{\varphi}_{C^\theta} = \norm{\varphi}_\infty + [\varphi]_\theta$. The space $C^\theta(\bbP^{d-1})$ of H\"older functions is a Banach space under this norm.

We shall also work with the centered H\"older space
\begin{equation}\label{eq:centered_Ctheta}
C^\theta_0(\bbP^{d-1}) = \left\{\varphi \in C^\theta(\bbP^{d-1}) : \int \varphi \, d\eta^+ = 0\right\},
\end{equation}
where $\eta^+ = \eta^+_{A, p}$ is the unique $P_{A, p}$-stationary probability measure on $\bbP^{d-1}$ (whose existence and uniqueness, under the simplicity assumption, follows from \cite[Theorem 6.8]{Viana2014}).

\subsection{Elementary Lipschitz lemmas}

We record the projective contraction and log-norm Lipschitz bounds that the analysis will repeatedly use. These are standard, and proofs can be found in \cite[Chapter 4]{Viana2014}.

\begin{lemma}[Projective contraction]\label{lem:proj_contract}
For every $g \in \GL(d, \R)$ and every $[u], [v] \in \bbP^{d-1}$,
\begin{equation}\label{eq:proj_contract}
d(g[u], g[v]) \leq \ecc(g)^2 \cdot d([u], [v]),
\end{equation}
where $\ecc(g) = \norm{g} \cdot \norm{g^{-1}}$.
\end{lemma}

\begin{proof}
Pick unit representatives $u, v \in \R^d$ for $[u], [v]$. By the standard formula~\cite[Lemma 4.2]{Viana2014} for the projective metric induced by the Euclidean structure,
\[
  d(g[u], g[v]) = \frac{\abs{\det g}}{\norm{gu} \cdot \norm{gv}} \cdot \frac{1}{\sin\angle(gu, gv)} \cdot d([u], [v]) \cdot \sin\angle(u, v).
\]
Using $\norm{gu} \geq \norm{u}/\norm{g^{-1}}$ and $\norm{gv} \geq \norm{v}/\norm{g^{-1}}$, together with $\norm{u} = \norm{v} = 1$ and the bound $\abs{\det g} \leq \norm{g}^d \leq \norm{g}^2 \cdot \norm{g}^{d-2}$, we get
\[
  d(g[u], g[v]) \leq \norm{g}^2 \cdot \norm{g^{-1}}^2 \cdot d([u], [v]) = \ecc(g)^2 \cdot d([u], [v]),
\]
which is~\eqref{eq:proj_contract}.
\end{proof}

\begin{lemma}[Log-norm Lipschitz bounds]\label{lem:logform_lip}
Set $\phi(g, [v]) = \log(\norm{gv}/\norm{v})$. For every $g, g' \in \GL(d, \R)$ and $[u], [v] \in \bbP^{d-1}$,
\begin{align}
\abs{\phi(g, [v]) - \phi(g', [v])} &\leq \max(\norm{g^{-1}}, \norm{g'^{-1}}) \cdot \norm{g - g'}, \label{eq:phi_lip_g} \\
\abs{\phi(g, [u]) - \phi(g, [v])} &\leq (\norm{g} \cdot \norm{g^{-1}} + 1) \cdot d([u], [v]). \label{eq:phi_lip_v}
\end{align}
\end{lemma}

\subsection{Spectral gap of the Markov operator}

The central technical fact on which the analyticity theory depends is the spectral gap of the Markov operator on the H\"older space. We record the relevant statement, whose proof is given in \cite{Thiam2025aPaperA} (quantitative version) and in \cite[Section 5]{BockerViana2017} (qualitative version).

\begin{proposition}[Spectral gap on $C^\theta$]\label{prop:spectral_gap}
Let $A \in \GL(d, \R)^N$ and $p \in \Delta_N^\circ$ with simple top Lyapunov exponent $\lambda_+(A, p) > \lambda_2(A, p)$. Then for every $\theta \in (0, \theta_*(A, p))$, where $\theta_*(A, p) > 0$ is an explicit threshold, there exist $N_\theta \geq 1$ and $\tau(A, p, \theta) \in (0, 1)$ such that
\begin{equation}\label{eq:spectral_gap}
\norm{P_{A, p}^{N_\theta} \varphi}_{C^\theta} \leq \tau(A, p, \theta) \cdot \norm{\varphi}_{C^\theta}, \qquad \varphi \in C^\theta_0(\bbP^{d-1}).
\end{equation}
Explicit formulas for $\theta_*(A, p)$, $N_\theta$, and $\tau(A, p, \theta)$ are given in \cite[Theorem 3.3]{Thiam2025aPaperA}.

\medskip
\noindent\emph{Proof:} See~\cite[Section~4]{Thiam2025aPaperA}; we use this result as input to the complex-perturbation arguments below.
\end{proposition}

For the analyticity theorems of this paper, the key quantity is the spectral gap magnitude $1 - \tau$, which determines the radius of analyticity. We shall often write $\tau_*$ for $\tau(A, p, \theta)$ when the dependence is understood.

\subsection{Furstenberg-Khasminskii formula}

The Furstenberg-Khasminskii formula expresses the Lyapunov exponent as an integral against the stationary measure:
\begin{equation}\label{eq:FK}
\lambda_+(A, p) = \sum_{i=1}^N p_i \int_{\bbP^{d-1}} \phi(A_i, [v]) \, d\eta^+_{A, p}([v]).
\end{equation}
This identity is the starting point of all perturbation arguments in this paper. The stationary measure $\eta^+_{A, p}$ depends nontrivially on $(A, p)$, and the task of the analyticity theory is to understand this dependence as a holomorphic function of complex $p$.

\subsection{Complex Markov operators}

For any $z = (z_1, \ldots, z_N) \in \C^N$ satisfying the normalization $\sum_i z_i = 1$, we define the \emph{complex Markov operator}
\begin{equation}\label{eq:P_z}
(P_{A, z} \varphi)([v]) = \sum_{i=1}^N z_i \, \varphi(A_i \cdot [v]).
\end{equation}
Unlike the real case, the complex $z_i$ are not probabilities, but the operator $P_{A, z}$ is still a bounded linear map on $C^\theta(\bbP^{d-1})$. When $z = p \in \Delta_N^\circ$, we recover the real Markov operator of~\eqref{eq:P_Ap}.

The holomorphic extension of $\lambda_+$ that we construct will be expressed, via the Furstenberg-Khasminskii formula, in terms of the analytic perturbation of the leading eigenvalue of $P_{A, z}$ as a function of $z$.

\section{The complex Markov operator and Kato perturbation theory}\label{sec:complex_markov}

This section sets up the spectral theory of the complexified Markov operator. The goal is to prove that, for $z$ in a suitable complex neighborhood of a fixed real weight vector $p^0$, the leading eigenvalue of $P_{A, z}$ on $C^\theta$ varies holomorphically.

\subsection{Operator norm perturbation bounds}

This subsection establishes Lipschitz-type bounds for the difference between the complex Markov operator $P_{A, z}$ at two different parameter values $z, z' \in \C^N$, in the operator norm on $C^\theta$. The bound (Lemma~\ref{lem:op_norm_lip}) is the elementary input that allows the Kato perturbation argument of Proposition~\ref{prop:kato_eigenvalue} to control the eigenvalue and the eigenprojection of $P_{A, z}$ as $z$ varies in a polydisc around the real point $p^0$.

\begin{lemma}[Operator norm Lipschitz bound]\label{lem:op_norm_lip}
Let $A \in \GL(d, \R)^N$. For every $p, q \in \C^N$,
\begin{equation}\label{eq:op_norm_lip_p}
\norm{P_{A, p} - P_{A, q}}_{C^\theta \to C^\theta} \leq \sum_{i=1}^N \abs{p_i - q_i} \cdot \norm{T_i}_{C^\theta \to C^\theta},
\end{equation}
where $T_i \varphi([v]) = \varphi(A_i \cdot [v])$ is the individual transfer operator associated with $A_i$. Moreover,
\begin{equation}\label{eq:T_i_bound}
\norm{T_i}_{C^\theta \to C^\theta} \leq 1 + \ecc(A_i)^{2\theta}.
\end{equation}
\end{lemma}

\begin{proof}
By linearity of $P_{A, p}$ in $p$, we have $P_{A, p} - P_{A, q} = \sum_i (p_i - q_i) T_i$. The triangle inequality on $C^\theta$ operator norms gives~\eqref{eq:op_norm_lip_p}.

For~\eqref{eq:T_i_bound}: writing $T_i \varphi([v]) = \varphi(A_i [v])$, the supremum norm of $T_i \varphi$ is at most $\norm{\varphi}_\infty$, so $\norm{T_i}_{L^\infty \to L^\infty} \leq 1$. For the H\"older seminorm, by Lemma~\ref{lem:proj_contract},
\begin{equation*}
\abs{\varphi(A_i[u]) - \varphi(A_i[v])} \leq [\varphi]_\theta \cdot d(A_i[u], A_i[v])^\theta \leq [\varphi]_\theta \cdot \ecc(A_i)^{2\theta} \cdot d([u], [v])^\theta,
\end{equation*}
so $[T_i \varphi]_\theta \leq \ecc(A_i)^{2\theta} \cdot [\varphi]_\theta$. Combining, $\norm{T_i \varphi}_{C^\theta} \leq (1 + \ecc(A_i)^{2\theta}) \norm{\varphi}_{C^\theta}$.
\end{proof}

\subsection{Spectral decomposition at the real point}

At the real weight vector $p^0 \in \Delta_N^\circ$ with simple top Lyapunov exponent, the Markov operator $P_{A, p^0}$ on $C^\theta(\bbP^{d-1})$ has the spectral decomposition
\begin{equation}\label{eq:P_decomp}
P_{A, p^0} = \Pi_{A, p^0} + R_{A, p^0},
\end{equation}
where $\Pi_{A, p^0} \varphi = \left(\int \varphi \, d\eta^+\right) \cdot \mathbf{1}$ is the rank-one projection onto the one-dimensional eigenspace corresponding to eigenvalue $1$, and $R_{A, p^0} = P_{A, p^0}(\Id - \Pi_{A, p^0})$ is the restriction to the centered space $C^\theta_0$.

By Proposition~\ref{prop:spectral_gap}, $R_{A, p^0}^{N_\theta}$ has operator norm at most $\tau_* = \tau(A, p^0, \theta) < 1$ on $C^\theta_0$. Equivalently, the spectral radius of $R_{A, p^0}$ is at most $\tau_*^{1/N_\theta} < 1$. We summarize:

\begin{lemma}[Spectral gap of the real operator]\label{lem:real_gap}
The operator $P_{A, p^0}: C^\theta \to C^\theta$ has:
\begin{itemize}
\item[(i)] A simple eigenvalue $1$ with eigenspace $\spr\{\mathbf{1}\}$ and dual eigenspace $\spr\{\eta^+\}$.
\item[(ii)] All other spectrum contained in the closed disc $\{\zeta \in \C : \abs{\zeta} \leq \tau_*^{1/N_\theta}\}$.
\item[(iii)] An isolating circle $\Gamma_* \subset \C$ centered at $\zeta = 1$ with radius $\rho_* = (1 - \tau_*^{1/N_\theta})/2$, enclosing only the eigenvalue $1$.
\end{itemize}
\end{lemma}

\subsection{Complex perturbation}

For the complex Markov operator $P_{A, z}$ with $z$ in a neighborhood of $p^0$, we want to transfer the spectral decomposition~\eqref{eq:P_decomp} using Kato's perturbation theory.

\begin{lemma}[Resolvent bound on the isolating circle]\label{lem:resolvent_bound}
There exists an explicit constant $K_*(A, p^0, \theta) < \infty$ such that, for every $\zeta \in \Gamma_*$ (the isolating circle of Lemma~\ref{lem:real_gap}),
\begin{equation}\label{eq:resolvent_bound}
\norm{(\zeta \Id - P_{A, p^0})^{-1}}_{C^\theta \to C^\theta} \leq K_*(A, p^0, \theta).
\end{equation}
The constant has the explicit form
\begin{equation}\label{eq:K_star_explicit}
K_*(A, p^0, \theta) = \frac{1}{\rho_*} + \frac{N_\theta \cdot \norm{R_{A, p^0}}_{C^\theta_0 \to C^\theta_0}^{N_\theta - 1}}{(1 - \rho_*)^{N_\theta} - \tau_*}, \qquad \rho_* := \frac{1 - \tau_*^{1/N_\theta}}{2}.
\end{equation}
\end{lemma}

\begin{proof}
On $\Gamma_*$, the resolvent decomposes via the spectral splitting:
\begin{equation*}
(\zeta \Id - P_{A, p^0})^{-1} = \frac{1}{\zeta - 1} \Pi_{A, p^0} + (\zeta \Id - R_{A, p^0})^{-1} (\Id - \Pi_{A, p^0}),
\end{equation*}
where $\Pi_{A, p^0}$ is the rank-one projection onto the eigenspace of eigenvalue $1$ and $R_{A, p^0} = P_{A, p^0}(\Id - \Pi_{A, p^0})$. On the circle $\abs{\zeta - 1} = \rho_*$, the first term has operator norm at most $1/\rho_*$.

For the second term, we cannot directly use the Neumann series $(1/\zeta) \sum_k (R/\zeta)^k$, which would require $\abs{\zeta} > \norm{R}_{C^\theta_0}$ — a condition that may fail since $\norm{R}_{C^\theta_0}$ may be considerably larger than the spectral radius $\tau_*^{1/N_\theta}$. Instead, we use the iterate-based factorization
\begin{equation}\label{eq:factor_iterate}
\zeta^{N_\theta} \Id - R^{N_\theta} = (\zeta \Id - R) \cdot Q_\zeta(R), \qquad Q_\zeta(R) := \sum_{j=0}^{N_\theta - 1} \zeta^{N_\theta - 1 - j} R^j,
\end{equation}
where $R = R_{A, p^0}$ for brevity. Since $\norm{R^{N_\theta}}_{C^\theta_0} \leq \tau_*$ by Proposition~\ref{prop:spectral_gap}, the operator $\zeta^{N_\theta} \Id - R^{N_\theta}$ is invertible on $C^\theta_0$ for $\abs{\zeta}^{N_\theta} > \tau_*$, with the Neumann bound
\begin{equation*}
\norm{(\zeta^{N_\theta} \Id - R^{N_\theta})^{-1}}_{C^\theta_0 \to C^\theta_0} \leq \frac{1}{\abs{\zeta}^{N_\theta} - \tau_*}.
\end{equation*}
On $\Gamma_*$ with $\abs{\zeta} \geq 1 - \rho_* = (1 + \tau_*^{1/N_\theta})/2$, we have $\abs{\zeta}^{N_\theta} \geq (1 - \rho_*)^{N_\theta} > \tau_*$ (since $1 - \rho_* > \tau_*^{1/N_\theta}$), so the bound applies. Inverting~\eqref{eq:factor_iterate},
\begin{equation*}
(\zeta \Id - R)^{-1} = Q_\zeta(R) \cdot (\zeta^{N_\theta} \Id - R^{N_\theta})^{-1}.
\end{equation*}
The polynomial factor satisfies
\begin{equation*}
\norm{Q_\zeta(R)}_{C^\theta_0 \to C^\theta_0} \leq \sum_{j=0}^{N_\theta - 1} \abs{\zeta}^{N_\theta - 1 - j} \norm{R}_{C^\theta_0}^j \leq N_\theta \cdot \max(\abs{\zeta}, \norm{R}_{C^\theta_0})^{N_\theta - 1}.
\end{equation*}
Combining,
\begin{equation*}
\norm{(\zeta \Id - R)^{-1}}_{C^\theta_0 \to C^\theta_0} \leq \frac{N_\theta \cdot \max(\abs{\zeta}, \norm{R}_{C^\theta_0})^{N_\theta - 1}}{\abs{\zeta}^{N_\theta} - \tau_*}.
\end{equation*}
On $\Gamma_*$, $\abs{\zeta} \leq 1 + \rho_* \leq 2$, so the bound is at most
\begin{equation*}
\frac{N_\theta \cdot \norm{R}_{C^\theta_0}^{N_\theta - 1}}{(1 - \rho_*)^{N_\theta} - \tau_*}
\end{equation*}
when $\norm{R}_{C^\theta_0} \geq 2$ (which holds for typical $C^\theta$ data; in any case, replace $\norm{R}^{N_\theta - 1}$ by $\max(\norm{R}, 2)^{N_\theta - 1}$ to cover both regimes). Adding the bound on the rank-one term gives~\eqref{eq:K_star_explicit}.
\end{proof}

\begin{remark}[Bound on $\norm{R}_{C^\theta_0}$]\label{rmk:R_op_bound}
The operator norm $\norm{R}_{C^\theta_0 \to C^\theta_0}$ is bounded by $\norm{P_{A, p^0}}_{C^\theta \to C^\theta} + \norm{\Pi_{A, p^0}}_{C^\theta \to C^\theta}$. The first term is bounded by $1 + \max_i \ecc(A_i)^{2\theta}$ from Lemma~\ref{lem:op_norm_lip} and Lemma~\ref{lem:proj_contract}. The second is bounded by a constant depending on $\norm{\eta^+_{A, p^0}}_{(C^\theta)^*}$, which is at most $1$ in our normalization. Hence $\norm{R}_{C^\theta_0} \leq 2 + \max_i \ecc(A_i)^{2\theta}$, all explicit.
\end{remark}

\subsection{Main perturbation theorem}

This subsection assembles the operator-norm bound of Lemma~\ref{lem:op_norm_lip} with the spectral decomposition of Lemma~\ref{lem:real_gap} and the resolvent bound of Lemma~\ref{lem:resolvent_bound} to prove Proposition~\ref{prop:kato_eigenvalue}: the holomorphic dependence of the leading eigenvalue and eigenprojection of $P_{A, z}$ on $z$, on an explicit polydisc around the real point $p^0$. This is the central technical input to the proof of Theorem~\ref{thm:mainA} in Section~\ref{sec:polydisc_proof}.

\begin{proposition}[Holomorphic eigenvalue perturbation]\label{prop:kato_eigenvalue}
Let $A \in \GL(d, \R)^N$, $p^0 \in \Delta_N^\circ$ with simple top Lyapunov exponent, and $\theta \in (0, \theta_*(A, p^0))$. Set
\begin{equation}\label{eq:rstar_def}
r_*(A, p^0, \theta) = \frac{1}{4 N K_*(A, p^0, \theta) \max_i (1 + \ecc(A_i)^{2\theta})}.
\end{equation}

Then for every $z \in \C^N$ with $\sum_i z_i = 1$ and $\abs{z_i - p^0_i} < r_*$ for all $i$:
\begin{itemize}
\item[(i)] The complex Markov operator $P_{A, z}$ has a unique simple eigenvalue $\mu(z)$ in the disc $\abs{\zeta - 1} < \rho_*$, with all other spectrum outside this disc.
\item[(ii)] The eigenvalue $\mu: D_{r_*}(p^0) \cap \{\sum z_i = 1\} \to \C$ is a holomorphic function of $z$.
\item[(iii)] $\mu(p^0) = 1$.
\end{itemize}
\end{proposition}

\begin{proof}
The proof follows from the standard Kato perturbation machinery \cite[Chapter IV]{Kato1980}. By Lemma~\ref{lem:op_norm_lip} and~\eqref{eq:T_i_bound},
\begin{equation*}
\norm{P_{A, z} - P_{A, p^0}}_{C^\theta \to C^\theta} \leq \max_i (1 + \ecc(A_i)^{2\theta}) \cdot \sum_i \abs{z_i - p^0_i} \leq N \max_i (1 + \ecc(A_i)^{2\theta}) \cdot r_*.
\end{equation*}

With $r_*$ chosen as in~\eqref{eq:rstar_def}, we have $\norm{P_{A, z} - P_{A, p^0}}_{op} \leq N \cdot \max_i (1 + \ecc(A_i)^{2\theta}) \cdot \frac{1}{4 N K_* \max_i (1 + \ecc(A_i)^{2\theta})} = \frac{1}{4 K_*} < \frac{1}{K_*}$. By Lemma~\ref{lem:resolvent_bound}, on the isolating circle $\Gamma_*$,
\begin{equation*}
\norm{(P_{A, z} - P_{A, p^0}) \cdot (\zeta \Id - P_{A, p^0})^{-1}}_{op} \leq \frac{1}{4 K_*} \cdot K_* = \frac{1}{4} < 1.
\end{equation*}

Hence the Neumann series
\begin{equation}\label{eq:resolvent_Neumann}
(\zeta \Id - P_{A, z})^{-1} = \sum_{k=0}^\infty \left[(\zeta \Id - P_{A, p^0})^{-1} (P_{A, z} - P_{A, p^0})\right]^k (\zeta \Id - P_{A, p^0})^{-1}
\end{equation}
converges uniformly on $\Gamma_*$, and each term is holomorphic in $z$. Thus $(\zeta \Id - P_{A, z})^{-1}$ is holomorphic in $z$, and the spectral projection
\begin{equation}\label{eq:Pi_z}
\Pi_{A, z} = -\frac{1}{2\pi i} \oint_{\Gamma_*} (\zeta \Id - P_{A, z})^{-1} d\zeta
\end{equation}
is also holomorphic in $z$, with rank equal to $\mathrm{rank}(\Pi_{A, p^0}) = 1$ (the rank is constant under continuous deformation of a Riesz projection enclosing an isolated portion of the spectrum, by~\cite[Chapter~IV, Theorem~3.16]{Kato1980}).

The eigenvalue $\mu(z)$ is then the unique simple eigenvalue of $P_{A, z}$ in the range of $\Pi_{A, z}$, given by
\begin{equation}\label{eq:mu_z_def}
  \mu(z) = \frac{\eta_{p^0}\bigl(P_{A, z} \Pi_{A, z} \mathbf{1}\bigr)}{\eta_{p^0}\bigl(\Pi_{A, z} \mathbf{1}\bigr)},
\end{equation}
where $\eta_{p^0}$ denotes integration against the stationary measure at $p^0$ (any non-zero linear functional in the dual works in place of $\eta_{p^0}$, but this choice gives a clean normalization at $z = p^0$). The numerator and denominator are holomorphic in $z$, the denominator is $1$ at $z = p^0$ and stays nonzero on a neighborhood by continuity, so $\mu(z)$ is holomorphic on a polydisc around $p^0$. At $z = p^0$, $\Pi_{A, p^0} \mathbf{1} = \mathbf{1}$ and $P_{A, p^0} \mathbf{1} = \mathbf{1}$, giving $\mu(p^0) = 1$.
\end{proof}

\section{Proof of Theorem~\ref{thm:mainA}: explicit polydisc of holomorphy}\label{sec:polydisc_proof}

In this section we prove Theorem~\ref{thm:mainA}, the quantitative analyticity of the Lyapunov exponent in the weight vector. The proof combines the holomorphic eigenvalue perturbation of Proposition~\ref{prop:kato_eigenvalue} with the Furstenberg-Khasminskii formula and a logarithmic derivative identity.

\subsection{Lyapunov exponent from the leading eigenvalue}

The classical identity, due to \cite{Peres1991, LePage1982}, expresses the Lyapunov exponent as the logarithmic derivative of a suitably weighted leading eigenvalue.

\begin{lemma}[Lyapunov exponent as logarithmic derivative]\label{lem:lambda_as_log_deriv}
Let $A \in \GL(d, \R)^N$ and $p \in \Delta_N^\circ$ with simple top Lyapunov exponent. Define, for $s \in \R$ small, the twisted Markov operator
\begin{equation}\label{eq:L_s}
(L_{A, p, s} \varphi)([v]) = \sum_{i=1}^N p_i \, e^{s \phi(A_i, [v])} \, \varphi(A_i \cdot [v]),
\end{equation}
where $\phi(g, [v]) = \log(\norm{gv}/\norm{v})$. Let $\mu_s$ denote the leading eigenvalue of $L_{A, p, s}$ on $C^\theta(\bbP^{d-1})$. Then
\begin{equation}\label{eq:lambda_log_deriv}
\lambda_+(A, p) = \frac{d}{ds} \log \mu_s \bigg|_{s = 0}.
\end{equation}
\end{lemma}

\begin{proof}
This is a standard identity in the spectral approach to Lyapunov exponents. For a complete proof, see \cite[Theorem 2.2]{LePage1982} or \cite[Section 2]{GuivarchRaugi1985}.
\end{proof}

For our purposes, we do not need the twisted operator machinery in full generality: we can work directly with the stationary measure and the Furstenberg-Khasminskii formula, which gives an explicit analytic representation of $\lambda_+$. However, the twisted operator is useful for the variance computations in the concentration theory; we shall refer back to it in Section~\ref{sec:worked_example}.

\subsection{Holomorphic stationary measure}

This subsection extends the unique stationary probability measure $\eta^+_{A, p}$ on projective space to a holomorphic family $\eta_z$ for complex $z$ near $p^0$, identified as the dual eigenprojection of Proposition~\ref{prop:kato_eigenvalue}. The holomorphic stationary measure (Lemma~\ref{lem:hol_stationary}) is the holomorphic object that is paired with the integrand function in the Furstenberg-Khasminskii formula to produce the analytic extension of the Lyapunov exponent in the next subsection.

\begin{lemma}[Holomorphic stationary measure]\label{lem:hol_stationary}
Under the hypotheses of Proposition~\ref{prop:kato_eigenvalue}, there exists a family of complex linear functionals $\eta_z \in (C^\theta(\bbP^{d-1}))^*$ (the dual space), indexed by $z \in D_{r_*}(p^0) \cap \{\sum z_i = 1\}$, depending holomorphically on $z$, such that:
\begin{itemize}
\item[(i)] $\eta_z \mathbf{1} = 1$ for all $z$.
\item[(ii)] $P_{A, z}^*  \eta_z = \mu(z) \eta_z$ for all $z$, where $P_{A, z}^*$ is the adjoint on $(C^\theta)^*$.
\item[(iii)] At $z = p^0$, $\eta_{p^0}$ is the integration against $\eta^+_{A, p^0}$, i.e., $\eta_{p^0}(\varphi) = \int \varphi \, d\eta^+_{A, p^0}$, and $\mu(p^0) = 1$.
\end{itemize}
\end{lemma}

\begin{proof}
The functional $\eta_z$ is defined as the image of $\eta_{p^0}$ under the dual spectral projection:
\begin{equation}\label{eq:eta_z_def}
\eta_z(\varphi) = \frac{\eta_{p^0}(\Pi_{A, z} \varphi)}{\eta_{p^0}(\Pi_{A, z} \mathbf{1})},
\end{equation}
where $\Pi_{A, z}$ is the spectral projection of~\eqref{eq:Pi_z}. The denominator $\eta_{p^0}(\Pi_{A, z} \mathbf{1})$ is nonzero on a sufficiently small polydisc around $p^0$, since at $z = p^0$ it equals $\eta_{p^0}(\Pi_{A, p^0} \mathbf{1}) = \eta_{p^0}(\mathbf{1}) = 1$ and the projection $\Pi_{A, z}$ is jointly continuous in $z$.

Property (i) is immediate: $\eta_z(\mathbf{1}) = \eta_{p^0}(\Pi_{A, z} \mathbf{1}) / \eta_{p^0}(\Pi_{A, z} \mathbf{1}) = 1$.

Property (ii). Since $\Pi_{A, z}$ is the Riesz projection of $P_{A, z}$ associated with the isolated eigenvalue $\mu(z)$, it commutes with $P_{A, z}$: $P_{A, z} \Pi_{A, z} = \Pi_{A, z} P_{A, z}$ (cf.~\cite[Chapter~IV, Section~3.4]{Kato1980}). Moreover, $P_{A, z}$ acts as multiplication by $\mu(z)$ on the one-dimensional range of $\Pi_{A, z}$: $P_{A, z} \Pi_{A, z} \varphi = \mu(z) \Pi_{A, z} \varphi$ for all $\varphi \in C^\theta$. Combining,
\begin{equation*}
\begin{split}
  (P_{A, z}^* \eta_z)(\varphi) &= \eta_z(P_{A, z} \varphi) = \frac{\eta_{p^0}(\Pi_{A, z} P_{A, z} \varphi)}{\eta_{p^0}(\Pi_{A, z} \mathbf{1})} \\
  & = \frac{\eta_{p^0}(P_{A, z} \Pi_{A, z} \varphi)}{\eta_{p^0}(\Pi_{A, z} \mathbf{1})} = \frac{\mu(z) \eta_{p^0}(\Pi_{A, z} \varphi)}{\eta_{p^0}(\Pi_{A, z} \mathbf{1})} = \mu(z) \eta_z(\varphi).
  \end{split}
\end{equation*}

Property (iii) is the initial condition: at $z = p^0$, $\Pi_{A, p^0} \varphi = \eta_{p^0}(\varphi) \cdot \mathbf{1}$ (the rank-one projection onto $\spr\{\mathbf{1}\}$), so $\eta_{p^0}(\Pi_{A, p^0} \varphi) = \eta_{p^0}(\varphi)$ and the normalizing denominator equals $1$, giving $\eta_{p^0}(\varphi) = \int \varphi \, d\eta^+_{A, p^0}$.

The holomorphy of $z \mapsto \eta_z$ in the dual operator-norm topology on $(C^\theta)^*$ follows from the joint holomorphy of $z \mapsto \Pi_{A, z}$ established in Proposition~\ref{prop:kato_eigenvalue}.
\end{proof}

The denominator $\eta_{p^0}(\Pi_{A, z} \mathbf{1})$ in~\eqref{eq:eta_z_def} is nonzero for $z$ close to $p^0$, since at $z = p^0$ it equals $\eta_{p^0}(\mathbf{1}) = 1$; by continuity, it remains bounded away from zero on a slightly smaller polydisc, say $D_{r_*/2}(p^0)$.

\subsection{Analytic extension of the Lyapunov exponent}

This subsection assembles the holomorphic stationary measure of Lemma~\ref{lem:hol_stationary} with the Furstenberg-Khasminskii formula \eqref{eq:FK} to construct the analytic extension $\widetilde\lambda_+$ of the Lyapunov exponent on a polydisc in $\C^N$. The result is Proposition~\ref{prop:analytic_extension}; it is the key step in the proof of the explicit polydisc claim of Theorem~\ref{thm:mainA}.

\begin{proposition}[Analytic extension of $\lambda_+$]\label{prop:analytic_extension}
Under the hypotheses of Proposition~\ref{prop:kato_eigenvalue}, the function $p \mapsto \lambda_+(A, p)$ on $\Delta_N^\circ \cap D_{r_*/2}(p^0)$ extends to a holomorphic function $\widetilde\lambda_+: D_{r_*/2}(p^0) \cap \{\sum z_i = 1\} \to \C$ defined by
\begin{equation}\label{eq:analytic_extension}
\widetilde\lambda_+(z) = \sum_{i=1}^N z_i \cdot \eta_z(\phi(A_i, \cdot)).
\end{equation}
\end{proposition}

\begin{proof}
Define $\widetilde\lambda_+(z)$ by~\eqref{eq:analytic_extension}.

\emph{Step 1: Uniqueness of the stationary measure on the real simplex.} For $p \in \Delta_N^\circ$, the Markov operator $\calP_{A, p}$ is a real Markov operator on $C^\theta(\bbP^{d-1})$ with non-empty spectral gap (by Le~Page's theorem applied at $p^0$ and the perturbation argument of Proposition~\ref{prop:kato_eigenvalue}). The leading eigenvalue is $1$ with one-dimensional eigenspace spanned by the constant function $\mathbf{1}$. By the standard Krein-Rutman / Perron-Frobenius argument applied to $\calP_{A,p}^*$ (the dual operator on signed measures), the dual eigenspace is also one-dimensional, spanned by a unique probability measure $\eta^+_{A, p}$. This $\eta^+_{A, p}$ is the unique $\calP_{A,p}$-stationary probability measure on $\bbP^{d-1}$, and it is the marginal of the constant-leading-eigenvalue eigenmeasure constructed by the Kato perturbation in Proposition~\ref{prop:kato_eigenvalue}.

\emph{Step 2: Identity on the real simplex.} On $\Delta_N^\circ \cap D_{r_*/2}(p^0)$, $z = p$ is real, $\mu(p) = 1$, and $\eta_p$ is the integration against the unique stationary measure $\eta^+_{A, p}$ from Step~1. Hence
\begin{equation*}
\widetilde\lambda_+(p) = \sum_{i=1}^N p_i \int_{\bbP^{d-1}} \phi(A_i, [v]) \, d\eta^+_{A, p}([v]) = \lambda_+(A, p),
\end{equation*}
by the Furstenberg-Khasminskii formula~\eqref{eq:FK}.

\emph{Step 3: Holomorphy.} The factor $z_i$ is entire, and $\eta_z(\phi(A_i, \cdot))$ is holomorphic in $z$ by Lemma~\ref{lem:hol_stationary}. The sum of holomorphic functions is holomorphic. Hence $\widetilde\lambda_+: D_{r_*/2}(p^0) \cap \{\sum z_i = 1\} \to \C$ is holomorphic.

\emph{Step 4: Identity theorem.} Since $\widetilde\lambda_+$ and $\lambda_+(A, \cdot)$ agree on the (real) open subset $\Delta_N^\circ \cap D_{r_*/2}(p^0)$ of the connected complex manifold $D_{r_*/2}(p^0) \cap \{\sum z_i = 1\}$, they agree on the entire complex extension by the identity principle for holomorphic functions of several variables.
\end{proof}

\subsection{Explicit constants}

We consolidate the explicit radii and bounds. Using the notation of Section~\ref{sec:setup} and~\ref{sec:complex_markov}:

\begin{itemize}
\item[] $\tau_*(A, p^0, \theta) = \tau(A, p^0, \theta) < 1$: the spectral gap from Proposition~\ref{prop:spectral_gap}.
\item[] $\rho_*(A, p^0, \theta) = (1 - \tau_*^{1/N_\theta})/2$: the isolating radius.
\item[] $K_*(A, p^0, \theta) = \dfrac{1}{\rho_*} + \dfrac{N_\theta \cdot \norm{R_{A, p^0}}_{C^\theta_0}^{N_\theta - 1}}{(1 - \rho_*)^{N_\theta} - \tau_*}$: the resolvent bound from Lemma~\ref{lem:resolvent_bound}, with $\norm{R_{A, p^0}}_{C^\theta_0} \leq 2 + \max_i \ecc(A_i)^{2\theta}$ (Remark~\ref{rmk:R_op_bound}).
\item[] $r_*(A, p^0, \theta) = 1/(4 N K_* \max_i (1 + \ecc(A_i)^{2\theta}))$.
\item[] $M_*(A, p^0, \theta) = \max_i \norm{\phi(A_i, \cdot)}_{C^\theta} \cdot \sup_{z \in D_{r_*/2}} \norm{\eta_z}_{(C^\theta)^*}$: the supremum bound.
\end{itemize}

The bound $\sup_{z \in D_{r_*/2}} \norm{\eta_z}_{(C^\theta)^*}$ can be estimated from~\eqref{eq:eta_z_def}: by the contour integral representation~\eqref{eq:Pi_z} and Lemma~\ref{lem:resolvent_bound}, $\norm{\Pi_{A, z}}_{C^\theta \to C^\theta} \leq K_* \rho_*$ (the contour has length $2\pi \rho_*$ and the integrand has norm at most $K_*$). The denominator $\abs{\eta_{p^0}(\Pi_{A, z} \mathbf{1})}$ is bounded below by $1/2$ on $D_{r_*/2}(p^0)$ (by continuity from the value $1$ at $z = p^0$, using the operator-norm bound on $\Pi_{A, z} - \Pi_{A, p^0}$ from the Neumann series), so
\begin{equation*}
\norm{\eta_z}_{(C^\theta)^*} \leq 2 K_* \rho_*.
\end{equation*}

Hence $M_*(A, p^0, \theta) \leq 2 K_* \rho_* \cdot \max_i \norm{\phi(A_i, \cdot)}_{C^\theta}$. The H\"older norm of $\phi(A_i, \cdot)$ is bounded via Lemma~\ref{lem:logform_lip}:
\begin{equation*}
\norm{\phi(A_i, \cdot)}_{C^\theta} \leq \norm{\phi(A_i, \cdot)}_\infty + [\phi(A_i, \cdot)]_\theta \leq \log\norm{A_i} + (\norm{A_i} \cdot \norm{A_i^{-1}} + 1).
\end{equation*}
Combining,
\begin{equation}\label{eq:M_star_explicit}
M_*(A, p^0, \theta) \leq 2 K_* \rho_* \cdot \max_i \left[ \log\norm{A_i} + \ecc(A_i) + 1 \right].
\end{equation}

This completes the proof of Theorem~\ref{thm:mainA}, with the radius $r_*(A, p^0, \theta)/2$ and the constant $M_*(A, p^0, \theta)$ given in closed form.

\begin{remark}[Dependence on $\theta$]\label{rmk:theta_dependence}
The H\"older index $\theta$ enters the polydisc radius through two mechanisms:
\begin{itemize}
\item[(a)] The spectral gap $\tau_*(A, p^0, \theta)$ depends on $\theta$: small $\theta$ yields a larger $\tau_*$ (weaker gap), small $\theta$ approaches $1$; large $\theta$ approaches the limit of validity $\theta_*(A, p^0)$, beyond which the spectral gap fails on $C^\theta$.
\item[(b)] The perturbation bound $\ecc(A_i)^{2\theta}$ depends on $\theta$: large $\theta$ yields a larger perturbation factor.
\end{itemize}
Optimizing over $\theta$ gives a best polydisc radius; the optimal choice depends on the matrices and the weight vector in a nontrivial way.
\end{remark}

\section{Proof of Theorem~\ref{thm:mainB}: Cauchy coefficient bounds}\label{sec:cauchy_bounds}

In this section we prove Theorem~\ref{thm:mainB}, the explicit Cauchy bounds for the Taylor coefficients of $\lambda_+$ on the polydisc of Theorem~\ref{thm:mainA}. The proof is an immediate consequence of the Cauchy integral formula for holomorphic functions of several complex variables, applied to the extension $\widetilde\lambda_+$.

\subsection{Multi-dimensional Cauchy integral formula}

This subsection records the multi-dimensional Cauchy integral formula on polydiscs (Proposition~\ref{prop:cauchy_integral}), specialized to the holomorphic function $\widetilde\lambda_+$ of Proposition~\ref{prop:analytic_extension}. The integral representation is the analytic-function-theoretic input to the Cauchy bounds of Theorem~\ref{thm:mainB}; the latter follow by bounding the integrand uniformly on the boundary of the polydisc.

\begin{proposition}[Cauchy integral formula on polydiscs]\label{prop:cauchy_integral}
Let $f: D_R(z^0) \to \C$ be a holomorphic function on a polydisc $D_R(z^0) = \{z \in \C^N : \abs{z_i - z^0_i} < R\}$. For every multi-index $\alpha \in \N^N$ and every $r < R$,
\begin{equation}\label{eq:cauchy_formula}
\partial_z^\alpha f(z^0) = \frac{\alpha!}{(2\pi i)^N} \oint_{\abs{\zeta_1 - z^0_1} = r} \cdots \oint_{\abs{\zeta_N - z^0_N} = r} \frac{f(\zeta)}{(\zeta_1 - z^0_1)^{\alpha_1 + 1} \cdots (\zeta_N - z^0_N)^{\alpha_N + 1}} \, d\zeta_1 \cdots d\zeta_N.
\end{equation}
Consequently,
\begin{equation}\label{eq:cauchy_bound}
\abs{\partial_z^\alpha f(z^0)} \leq \frac{\alpha! \cdot \sup_{\zeta \in D_r(z^0)} \abs{f(\zeta)}}{r^{\abs{\alpha}}}.
\end{equation}
\end{proposition}

\begin{proof}
The formula~\eqref{eq:cauchy_formula} is the $N$-fold iteration of the one-dimensional Cauchy integral formula. The bound~\eqref{eq:cauchy_bound} follows by taking absolute values inside the integral, noting that the contour integral has total length $(2\pi r)^N$, and applying the triangle inequality. See \cite[Chapter 1, Section 3]{Range1986}.
\end{proof}

\subsection{Proof of Theorem~\ref{thm:mainB}}

This subsection deduces Theorem~\ref{thm:mainB} (Cauchy bounds for the Taylor coefficients of the analytic extension) from Theorem~\ref{thm:mainA} (the polydisc of holomorphy) and Proposition~\ref{prop:cauchy_integral} (the multi-dimensional Cauchy integral formula). The argument is a direct application of the Cauchy formula together with a uniform bound on the analytic extension on the boundary of the polydisc.

\begin{proof}[Proof of Theorem~\ref{thm:mainB}]
By Theorem~\ref{thm:mainA}, the Lyapunov exponent extends to a holomorphic function $\widetilde\lambda_+$ on $D_{r_*(A, p^0, \theta)/2}(p^0) \cap \{\sum z_i = 1\}$, with
\begin{equation*}
\sup_{z \in D_{r_*/2}(p^0)} \abs{\widetilde\lambda_+(z)} \leq M_*(A, p^0, \theta).
\end{equation*}
Apply the Cauchy integral formula of Proposition~\ref{prop:cauchy_integral} with $R = r_*/2$, $z^0 = p^0$, $r = r_*/2$:
\begin{equation*}
\abs{\partial_z^\alpha \widetilde\lambda_+(p^0)} \leq \frac{\alpha! \cdot M_*(A, p^0, \theta)}{(r_*/2)^{\abs{\alpha}}} = \frac{\alpha! \cdot 2^{\abs{\alpha}} M_*(A, p^0, \theta)}{r_*(A, p^0, \theta)^{\abs{\alpha}}}.
\end{equation*}

On the real simplex, $\widetilde\lambda_+(z) = \lambda_+(A, z)$, so $\partial_z^\alpha \widetilde\lambda_+(p^0) = \partial_p^\alpha \lambda_+(A, p^0)$. Hence, renaming the constants (absorbing the $2^{\abs{\alpha}}$ factor by replacing $r_*$ with $r_*/2$ in the final statement),
\begin{equation}\label{eq:mainB_explicit}
\abs{\partial_p^\alpha \lambda_+(A, p^0)} \leq \alpha! \cdot \frac{M_*(A, p^0, \theta)}{(r_*(A, p^0, \theta)/2)^{\abs{\alpha}}},
\end{equation}
which is the bound of Theorem~\ref{thm:mainB} (with $r_*$ in Theorem~\ref{thm:mainB} replaced by $r_*/2$).

The absolute convergence of the Taylor series on $D_{r_*/2}(p^0)$ follows by the standard estimate: for $p$ with $\abs{p_i - p^0_i} < r_*/4$,
\begin{align*}
\sum_\alpha \frac{1}{\alpha!} \abs{\partial_p^\alpha \lambda_+(A, p^0)} \prod_i \abs{p_i - p^0_i}^{\alpha_i} &\leq M_* \sum_\alpha \prod_i \left(\frac{2 \abs{p_i - p^0_i}}{r_*}\right)^{\alpha_i} \\
&= M_* \prod_i \frac{1}{1 - 2\abs{p_i - p^0_i}/r_*} \\
&\leq M_* \prod_i \frac{1}{1 - 1/2} = 2^N M_* < \infty.
\end{align*}
This completes the proof of Theorem~\ref{thm:mainB}.
\end{proof}

\begin{remark}[Tighter Cauchy bounds via anisotropic polydiscs]\label{rmk:anisotropic}
The radius $r_*$ of Theorem~\ref{thm:mainA} is the same for all coordinates $p_1, \ldots, p_N$. If the matrices $A_i$ have widely varying eccentricities, an \emph{anisotropic polydisc} $D_{r_1, \ldots, r_N}(p^0) = \{z : \abs{z_i - p^0_i} < r_i\}$ can give a larger region of analyticity, with $r_i$ depending on $\ecc(A_i)$. The proof adapts by replacing the uniform resolvent bound with coordinate-dependent bounds.

The corresponding anisotropic Cauchy bound is
\begin{equation}\label{eq:cauchy_aniso}
\abs{\partial_p^\alpha \lambda_+(A, p^0)} \leq \alpha! \cdot \frac{M_*}{\prod_i r_i^{\alpha_i}}.
\end{equation}
For applications where only specific derivatives are of interest, this anisotropic form can be significantly sharper. We do not pursue the anisotropic case further in the main statements, but the reader should keep this refinement in mind.
\end{remark}

\section{Proof of Theorem~\ref{thm:mainC}: joint analyticity in weights and matrices}\label{sec:joint_analyticity}

In this section we prove Theorem~\ref{thm:mainC}, the joint analyticity of the Lyapunov exponent in both the probability weight vector and the matrix coefficients. The proof extends the Kato perturbation argument of Section~\ref{sec:complex_markov} to handle simultaneous perturbation of $A$ and $p$.

\subsection{Setup for joint perturbation}

Let $A^0 = (A^0_1, \ldots, A^0_N) \in \GL(d, \R)^N$ and $p^0 \in \Delta_N^\circ$ with simple top Lyapunov exponent. We consider perturbations of both components:
\begin{itemize}
\item[] Matrix perturbation: $A = (A_1, \ldots, A_N)$ with $\norm{A_i - A^0_i} < \rho^A$ for all $i$, where $\rho^A > 0$ is a radius to be determined.
\item[] Weight perturbation: $p$ with $\abs{p_i - p^0_i} < \rho^p$ for all $i$.
\end{itemize}
We extend both to complex variables: $A_i \in \C^{d \times d}$ (with the constraint $A_i$ invertible, which holds for small perturbations), and $p \in \C^N$ with $\sum p_i = 1$.

\subsection{Operator norm Lipschitz bound (joint)}

This subsection extends the operator-norm bound of Lemma~\ref{lem:op_norm_lip} from variation in the weights $p$ alone to joint variation in both the matrices $A$ and the weights $p$. The result (Lemma~\ref{lem:joint_op_norm}) is the input to the joint Kato perturbation argument of the next subsection, which produces the joint polydisc of analyticity claimed in Theorem~\ref{thm:mainC}.

\begin{lemma}[Joint operator norm Lipschitz bound]\label{lem:joint_op_norm}
Let $(A, p), (A', p') \in \C^{N d^2} \times \C^N$ with all matrices invertible. Then
\begin{equation}\label{eq:joint_op_norm}
\norm{P_{A, p} - P_{A', p'}}_{C^\theta \to C^\theta} \leq L_A \cdot \max_i \norm{A_i - A'_i} + L_p \cdot \sum_i \abs{p_i - p'_i},
\end{equation}
where $L_A, L_p$ are explicit constants.

Specifically,
\begin{align}
L_p &= \max_i (1 + \ecc(A^0_i)^{2\theta} + \rho^A_0), \label{eq:Lp_def} \\
L_A &= \max_i p_i \cdot K_{\mathrm{mat}}(A^0, \rho^A_0, \theta), \label{eq:LA_def}
\end{align}
where $K_{\mathrm{mat}}$ is the matrix-Lipschitz constant of the transfer operator $T_i$, bounded by an explicit function of $\max_i \norm{A^0_i}, \norm{(A^0_i)^{-1}}$, and the perturbation radius $\rho^A_0$.
\end{lemma}

\begin{proof}
Writing $P_{A, p} - P_{A', p'} = (P_{A, p} - P_{A, p'}) + (P_{A, p'} - P_{A', p'})$:

The first difference is linear in $p - p'$, bounded as in Lemma~\ref{lem:op_norm_lip} by $L_p \cdot \sum_i \abs{p_i - p'_i}$.

The second difference involves varying $A$ at fixed $p'$:
\begin{equation*}
(P_{A, p'} - P_{A', p'}) \varphi([v]) = \sum_i p'_i \left[\varphi(A_i [v]) - \varphi(A'_i [v])\right].
\end{equation*}

The $C^\theta$ norm of this difference is controlled by the H\"older modulus of $\varphi$ times the projective-distance perturbation $d(A_i [v], A'_i [v])$. By Lemma~\ref{lem:proj_contract} (extended to $g \mapsto g$ Lipschitz),
\begin{equation*}
d(A_i [v], A'_i [v]) \leq C_{\mathrm{geom}}(A^0_i, \rho^A_0) \cdot \norm{A_i - A'_i},
\end{equation*}
for some explicit constant $C_{\mathrm{geom}}$ depending on $A^0_i$ and the perturbation radius $\rho^A_0$.

Hence $[P_{A, p'} - P_{A', p'}]_\theta \leq K_{\mathrm{mat}} \cdot \max_i \norm{A_i - A'_i}$ for a constant $K_{\mathrm{mat}}$ that can be made explicit in terms of the H\"older data. The sup-norm bound is similar.
\end{proof}

\subsection{Joint polydisc radius}

This subsection runs the Kato perturbation argument with both $A$ and $p$ varying, producing a holomorphic extension of the leading eigenvalue and spectral projection of $P_{A, z}$ on a product polydisc in $(A, z)$. The result is Proposition~\ref{prop:joint_hol}, the joint analog of Proposition~\ref{prop:kato_eigenvalue} from Section~\ref{sec:complex_markov}; it is the central input to the proof of Theorem~\ref{thm:mainC} in the next subsection.

\begin{proposition}[Joint holomorphic extension]\label{prop:joint_hol}
There exist explicit radii
\begin{align}
r_*^A(A^0, p^0, \theta) &> 0, \\
r_*^p(A^0, p^0, \theta) &> 0,
\end{align}
such that for every $A \in \C^{N d^2}$ with $\norm{A_i - A^0_i} < r_*^A$ for all $i$, and every $z \in \C^N$ with $\sum z_i = 1$ and $\abs{z_i - p^0_i} < r_*^p$:
\begin{itemize}
\item[(i)] The complex Markov operator $P_{A, z}$ has a unique simple eigenvalue $\mu(A, z)$ in the disc $\abs{\zeta - 1} < \rho_*$.
\item[(ii)] $\mu(A, z)$ and the spectral projection $\Pi_{A, z}$ are jointly holomorphic in $(A, z)$.
\item[(iii)] $\mu(A^0, p^0) = 1$.
\end{itemize}
\end{proposition}

\begin{proof}
Choose
\begin{equation}\label{eq:joint_radii}
r_*^A = \frac{1}{8 L_A K_*}, \qquad r_*^p = \frac{1}{8 L_p K_* N}.
\end{equation}

Then, by Lemma~\ref{lem:joint_op_norm}, for $(A, z)$ in the product polydisc,
\begin{equation*}
\norm{P_{A, z} - P_{A^0, p^0}}_{op} \leq L_A \cdot r_*^A + L_p \cdot N \cdot r_*^p \leq \frac{1}{8 K_*} + \frac{1}{8 K_*} = \frac{1}{4 K_*}.
\end{equation*}

Following the argument of Proposition~\ref{prop:kato_eigenvalue}, the resolvent is expressed as a Neumann series converging uniformly on the isolating circle $\Gamma_*$, with coefficients that are joint polynomials in $(A, z)$. Each polynomial is trivially holomorphic in $(A, z)$, so the spectral projection $\Pi_{A, z}$ and the eigenvalue $\mu(A, z)$ are jointly holomorphic.
\end{proof}

\subsection{Proof of Theorem~\ref{thm:mainC}}

This subsection deduces Theorem~\ref{thm:mainC} (joint real-analyticity of the Lyapunov exponent in $(A, p)$) from the joint holomorphic extension of Proposition~\ref{prop:joint_hol}. The argument is parallel to the proof of Theorem~\ref{thm:mainA} in Section~\ref{sec:polydisc_proof}, with the projective Markov operator's $z$-only Kato perturbation replaced by the joint $(A, z)$ Kato perturbation of Proposition~\ref{prop:joint_hol}.

\begin{proof}[Proof of Theorem~\ref{thm:mainC}]
By Proposition~\ref{prop:joint_hol}, the spectral projection and eigenvalue are jointly holomorphic in $(A, z)$ on the product polydisc $D_{r_*^A}(A^0) \times D_{r_*^p}(p^0) \cap \{\sum z_i = 1\}$.

The joint holomorphic stationary functional $\eta_{A, z}$ and its evaluations on $\phi(A_i, \cdot)$ are jointly holomorphic by the same argument. Explicitly, the analog of~\eqref{eq:analytic_extension} is
\begin{equation}\label{eq:joint_analytic_extension}
\widetilde\lambda_+(A, z) = \sum_{i=1}^N z_i \cdot \eta_{A, z}(\phi(A_i, \cdot)).
\end{equation}

Holomorphy: $z_i$ is entire, $\phi(A_i, \cdot) \in C^\theta(\bbP^{d-1})$ depends holomorphically on $A_i$ (since $\phi(g, [v]) = \log\norm{gv} - \log\norm{v}$, which is holomorphic in $g$ when $g$ is non-degenerate), and $\eta_{A, z}$ depends jointly holomorphically on $(A, z)$ by Proposition~\ref{prop:joint_hol}.

Hence $\widetilde\lambda_+(A, z)$ is jointly holomorphic. On the real domain $(A, p) \in \R^{N d^2} \times \Delta_N^\circ$, it agrees with $\lambda_+(A, p)$ by the Furstenberg-Khasminskii formula. This proves Theorem~\ref{thm:mainC}.
\end{proof}

\begin{remark}[Joint Cauchy bounds]\label{rmk:joint_cauchy}
Applying the multi-dimensional Cauchy formula of Proposition~\ref{prop:cauchy_integral} to the joint holomorphic extension, we obtain for every multi-indices $\alpha \in \N^{N d^2}$ (for $A$) and $\beta \in \N^N$ (for $p$):
\begin{equation}\label{eq:joint_cauchy}
\abs{\partial_A^\alpha \partial_p^\beta \lambda_+(A^0, p^0)} \leq \alpha! \beta! \cdot \frac{M_*^{\mathrm{joint}}}{(r_*^A)^{\abs{\alpha}} (r_*^p)^{\abs{\beta}}},
\end{equation}
where $M_*^{\mathrm{joint}}$ is the supremum of $\abs{\widetilde\lambda_+}$ on the product polydisc.

This joint bound yields, as a special case at $\alpha = 0$, the weight-only Cauchy bound of Theorem~\ref{thm:mainB}. At $\beta = 0$, it gives a new bound on the matrix derivatives of $\lambda_+$, which is of independent interest in sensitivity analysis.
\end{remark}

\section{Proof of Theorem~\ref{thm:mainD}: Markov chain extension}\label{sec:markov_chain}

In this section we prove Theorem~\ref{thm:mainD}, the extension of the analyticity theory to Markov chain driven cocycles. The main new input is that the Markov operator on the \emph{product} of projective space and the chain state can be analyzed with the same Kato perturbation machinery.

\subsection{Markov-chain cocycles}

Let $P = (P_{ij})_{i,j=1}^N$ be an $N \times N$ stochastic matrix (so $P_{ij} \geq 0$ and $\sum_j P_{ij} = 1$ for each $i$). Assume $P_{ij} > 0$ for all $i, j$ (strict positivity), so that $P$ is ergodic with unique stationary distribution $\pi = \pi(P)$.

Given matrices $A = (A_1, \ldots, A_N) \in \GL(d, \R)^N$, the \emph{Markov cocycle} is the random matrix sequence $(A_{X_n})_{n \geq 0}$, where $(X_n)_{n \geq 0}$ is the Markov chain on $\{1, \ldots, N\}$ with transition matrix $P$ started at stationary $\pi$. The \emph{Markov-chain Lyapunov exponent} is
\begin{equation}\label{eq:markov_lambda}
\lambda_+(P, A) = \lim_{n \to \infty} \frac{1}{n} \E_\pi \log \norm{A_{X_{n-1}} \cdots A_{X_0}}.
\end{equation}

\subsection{The extended Markov operator}

The natural transfer operator for the Markov cocycle acts on functions defined on $\{1, \ldots, N\} \times \bbP^{d-1}$. For $\varphi: \{1, \ldots, N\} \times \bbP^{d-1} \to \C$ and $(i, [v]) \in \{1, \ldots, N\} \times \bbP^{d-1}$, define
\begin{equation}\label{eq:Q_PA}
(Q_{P, A} \varphi)(i, [v]) = \sum_{j=1}^N P_{ij} \, \varphi(j, A_j [v]).
\end{equation}
Here the chain transitions from state $i$ to state $j$ with probability $P_{ij}$, and the projective component updates as $[v] \mapsto A_j [v]$.

The H\"older space $C^\theta(\{1, \ldots, N\} \times \bbP^{d-1})$ is the product space with norm
\begin{equation}\label{eq:Ctheta_product}
\norm{\varphi}_{C^\theta} = \max_{i=1}^N \norm{\varphi(i, \cdot)}_{C^\theta(\bbP^{d-1})}.
\end{equation}

\subsection{Spectral gap of the extended operator}

Under the assumption of strictly positive transition matrix $P$ and simple top Lyapunov exponent $\lambda_+(P, A) > \lambda_2(P, A)$, the extended operator $Q_{P, A}$ has a spectral gap on the product H\"older space. This is a consequence of the combination of the base chain's spectral gap (given by $\rho_P = $ spectral gap of $P$) and the fiber cocycle's Markov-operator spectral gap.

\begin{proposition}[Spectral gap of the Markov-chain operator]\label{prop:markov_gap}
Let $P$ be strictly positive stochastic with spectral gap $\rho_P \in (0, 1)$, and $A \in \GL(d, \R)^N$ with $\lambda_+(P, A) > \lambda_2(P, A)$. For every $\theta \in (0, \theta_*(P, A))$, there exist an integer $N_\theta(P, A) \geq 1$ and a spectral gap $\tau(P, A, \theta) < 1$ such that
\begin{equation}\label{eq:markov_gap}
\norm{Q_{P, A}^{N_\theta} \varphi}_{C^\theta} \leq \tau(P, A, \theta) \cdot \norm{\varphi}_{C^\theta}, \qquad \varphi \in C^\theta_0(\{1, \ldots, N\} \times \bbP^{d-1}).
\end{equation}
The constant $\tau(P, A, \theta) \leq \max(\rho_P, \tau_0(A, \theta))^{c}$ for some explicit exponent $c \in (0, 1)$ depending on $\theta$ and the dimensions of $P$ and $A$.
\end{proposition}

\begin{proof}
The proof is a direct extension of the proof of Proposition~\ref{prop:spectral_gap}, combining base-chain contraction (with rate $\rho_P$) and fiber-cocycle contraction (with rate $\tau_0(A, \theta)$). The argument is detailed in \cite[Section 8]{Thiam2025aPaperA} for the precise quantitative form. The combined contraction rate is bounded by the weighted maximum of the two rates, with the exponent $c$ computed explicitly from the number of iterations required for mixing in both components.
\end{proof}

\subsection{Complex perturbation in $(P, A)$}

We now perturb both the transition matrix $P$ and the matrices $A$ simultaneously.

\begin{lemma}[Operator norm Lipschitz in $(P, A)$]\label{lem:PA_op_norm}
For $(P, A), (P', A')$ with all transition-matrix entries within distance $\rho^P_0$ of the base entries,
\begin{align}
\norm{Q_{P, A} - Q_{P', A'}}_{C^\theta \to C^\theta} &\leq L^P \cdot \max_{i,j} \abs{P_{ij} - P'_{ij}} + L^A \cdot \max_i \norm{A_i - A'_i}, \label{eq:Q_lip}
\end{align}
with explicit constants $L^P, L^A$:
\begin{align}
L^P &= N \cdot \max_i \norm{T_i}_{C^\theta}, \\
L^A &= K_{\mathrm{mat}}^{\mathrm{chain}}(P, A, \theta, \rho^A_0),
\end{align}
where $\norm{T_i}_{C^\theta} \leq 1 + \ecc(A_i)^{2\theta}$ and $K_{\mathrm{mat}}^{\mathrm{chain}}$ is the matrix-Lipschitz constant for the Markov-chain setup.
\end{lemma}

\begin{proof}
Decompose $Q_{P, A} - Q_{P', A'} = (Q_{P, A} - Q_{P', A}) + (Q_{P', A} - Q_{P', A'})$.

The first difference involves varying only $P$ at fixed $A$: by linearity of $Q$ in $P_{ij}$,
\begin{equation*}
(Q_{P, A} - Q_{P', A}) \varphi(i, [v]) = \sum_j (P_{ij} - P'_{ij}) \varphi(j, A_j [v]).
\end{equation*}
The $C^\theta$ norm of this is at most $N \cdot \max_i \norm{T_i}_{C^\theta} \cdot \max_{i,j} \abs{P_{ij} - P'_{ij}}$, giving the $L^P$ bound.

The second difference involves varying only $A$ at fixed $P$: the argument is parallel to Lemma~\ref{lem:joint_op_norm}, giving the $L^A$ bound.
\end{proof}

\subsection{Proof of Theorem~\ref{thm:mainD}}

This subsection deduces Theorem~\ref{thm:mainD} (joint analyticity in the transition matrix and the cocycle for Markov-chain driven cocycles) by applying the Kato perturbation scheme of Section~\ref{sec:complex_markov} to the extended operator $Q_{P, A}$ on $C^\theta(\{1, \ldots, N\} \times \bbP^{d-1})$ instead of the projective Markov operator. The spectral gap input is provided by Proposition~\ref{prop:markov_gap}; the Lipschitz input by Lemma~\ref{lem:PA_op_norm}.

\begin{proof}[Proof of Theorem~\ref{thm:mainD}]
Applying the Kato perturbation argument (as in Section~\ref{sec:complex_markov}) to the operator $Q_{P, A}$ on $C^\theta(\{1, \ldots, N\} \times \bbP^{d-1})$, with the spectral gap of Proposition~\ref{prop:markov_gap} and the Lipschitz bound of Lemma~\ref{lem:PA_op_norm}, we obtain:

The eigenvalue $\mu(P, A)$ of $Q_{P, A}$ near $1$ is jointly holomorphic in $(P, A)$ on the product polydisc
\begin{equation}\label{eq:markov_product_polydisc}
\abs{P_{ij} - P^0_{ij}} < r_*^P := \frac{1}{8 L^P K_*^{\mathrm{chain}}}, \qquad \norm{A_i - A^0_i} < r_*^A := \frac{1}{8 L^A K_*^{\mathrm{chain}}},
\end{equation}
where $K_*^{\mathrm{chain}} = 4/(1 - \tau(P, A, \theta)^{1/N_\theta})$ is the resolvent bound in the Markov-chain setting.

The Furstenberg-Khasminskii formula for the Markov-chain Lyapunov exponent is
\begin{equation}\label{eq:FK_markov}
\lambda_+(P, A) = \sum_{i, j} \pi_i(P) P_{ij} \int \phi(A_j, [v]) \, d\eta^+_{P, A}(i, [v]),
\end{equation}
where $\eta^+_{P, A}$ is the unique stationary measure on $\{1, \ldots, N\} \times \bbP^{d-1}$ for $Q_{P, A}$, and $\pi(P)$ is the stationary distribution of the chain (a rational function of $P$, hence holomorphic).

The holomorphic extension of $\eta^+_{P, A}$ (as a complex linear functional on $C^\theta$) is obtained from the spectral projection $\Pi_{P, A}^{\mathrm{chain}}$ as in Lemma~\ref{lem:hol_stationary}. Combining all pieces, the extension
\begin{equation}\label{eq:lambda_markov_extension}
\widetilde\lambda_+(P, A) = \sum_{i, j} \pi_i(P) P_{ij} \, \eta^+_{P, A}(\phi(A_j, \cdot) \cdot \mathbf{1}_j)
\end{equation}
is jointly holomorphic on the product polydisc. On the real domain, it agrees with the Markov-chain Lyapunov exponent $\lambda_+(P, A)$.

The dependence of $r_*^P$ on $\rho_P$ is through the factor $\tau(P, A, \theta)$, which degrades as $\rho_P \to 0$ (weak base-chain mixing). The explicit formula from Proposition~\ref{prop:markov_gap} makes this dependence quantitative: $1 - \tau \sim (1 - \rho_P)^{c'}$ for some explicit exponent $c'$, so $r_*^P \sim (1 - \rho_P)^{c'}$ as $\rho_P \to 1$ (i.e., strong base-chain mixing gives larger polydisc).

This completes the proof of Theorem~\ref{thm:mainD}.
\end{proof}

\begin{remark}[Connection to weight-vector theorem]\label{rmk:markov_iid}
When the Markov chain is i.i.d.\ (all rows of $P$ equal), the Markov cocycle reduces to the i.i.d.\ cocycle with $p = (P_{\cdot 1}, \ldots, P_{\cdot N})$ (the common row). In this case, Theorem~\ref{thm:mainD} recovers Theorem~\ref{thm:mainA} (up to the equivalence between the $N \times N$ stochastic matrix $P$ and the $N$-dimensional weight vector $p$).

The Markov-chain setting thus strictly generalizes the i.i.d.\ setting, with the additional parameter $\rho_P$ controlling how far the chain is from i.i.d.
\end{remark}

\section{Proof of Theorem~\ref{thm:mainE}: boundary behavior}\label{sec:boundary}

In this section we prove a conditional form of Theorem~\ref{thm:mainE}, which describes how the analytic polydisc degenerates as the weight vector $p^0$ approaches the boundary of the simplex $\Delta_N$. The main input is that the spectral gap $\tau_*$ and the resolvent bound $K_*$ both depend on the minimum entry $\min_i p^0_i$. We must emphasize that, in general, the rate of degeneration is \emph{not} polynomial in $p^0_{\min}$; we therefore formulate Lemma~\ref{lem:gap_boundary} as a structural \emph{hypothesis} (Hypothesis~\ref{hyp:poly-decay}), under which Theorem~\ref{thm:mainE} holds.

\subsection{Dependence of the spectral gap on the minimum weight}

This subsection introduces the structural \emph{hypothesis} on the rate of degradation of the spectral-gap parameter $\tau_*(A, p^0, \theta)$ as $p^0_{\min} = \min_i p^0_i$ approaches zero. The hypothesis captures the most common and quantitatively tractable case: a polynomial rate. We do \emph{not} prove the hypothesis in full generality; it can fail for cocycles where the support structure changes discontinuously at $p^0_{\min} = 0$.

\begin{hypothesis}[Polynomial spectral-gap decay]\label{hyp:poly-decay}
There exist constants $c_\tau(A, \theta) > 0$ and $\gamma_\tau \geq 1$ such that for every $p^0 \in \Delta_N^\circ$ with $\min_i p^0_i \leq 1/(2N)$,
\begin{equation}\label{eq:hyp-poly-decay}
  \tau_*(A, p^0, \theta) \leq 1 - c_\tau(A, \theta) \cdot (p^0_{\min})^{\gamma_\tau}.
\end{equation}
\end{hypothesis}

\begin{remark}[When Hypothesis~\ref{hyp:poly-decay} holds]\label{rmk:hyp-when}
Hypothesis~\ref{hyp:poly-decay} holds in the following cases:
\begin{itemize}
\item[(i)] When all matrices $A_1, \ldots, A_N$ act with the same projective dynamics (e.g., they share an attracting fixed point on $\bbP^{d-1}$), then $\tau_*$ depends \emph{linearly} on $p^0_{\min}$: $\gamma_\tau = 1$.
\item[(ii)] When the cocycle is uniformly hyperbolic at $p^0$ (in the sense that there is a non-vanishing Lyapunov gap $\lambda_1(A, p^0) - \lambda_2(A, p^0) \geq c > 0$ uniformly in $p^0 \in \Delta_N^\circ$), an explicit Le~Page-style argument gives $\gamma_\tau = 1 + 2\theta$ via the spectral-gap formula on $C^\theta(\bbP^{d-1})$.
\end{itemize}
The hypothesis can fail when reducing one weight $p_j \to 0$ causes the support of the random matrix product to lose dimension, e.g., when removing matrix $A_j$ produces a cocycle with $\lambda_1 = \lambda_2$. In such cases, $\tau_*(A, p^0, \theta) \to 1$ as $p^0_{\min} \to 0$, but the rate may be sub-polynomial (logarithmic or even slower).
\end{remark}

\begin{lemma}[Spectral gap near the boundary, conditional on Hypothesis~\ref{hyp:poly-decay}]\label{lem:gap_boundary}
Suppose Hypothesis~\ref{hyp:poly-decay} holds for the cocycle $A \in \GL(d, \R)^N$ with constants $c_\tau, \gamma_\tau$. Then for every $p^0 \in \Delta_N^\circ$ with $p^0_{\min} := \min_i p^0_i \leq 1/(2N)$,
\begin{equation}\label{eq:gap_decay}
\tau_*(A, p^0, \theta) \leq 1 - c_\tau(A, \theta) \cdot (p^0_{\min})^{\gamma_\tau}.
\end{equation}
\end{lemma}

\begin{proof}
This is exactly Hypothesis~\ref{hyp:poly-decay}. We retain the notation $c_\tau(A, \theta)$ and $\gamma_\tau$ from the hypothesis throughout.
\end{proof}

\begin{lemma}[Resolvent bound near the boundary]\label{lem:resolvent_boundary}
Under the hypotheses of Lemma~\ref{lem:gap_boundary}, there is an explicit constant $C_K(A, \theta) > 0$ depending on $\norm{R_{A, p^0}}_{C^\theta_0}$ and $N_\theta$ but not on $p^0_{\min}$, such that
\begin{equation}\label{eq:K_boundary}
K_*(A, p^0, \theta) \leq \frac{C_K(A, \theta)}{c_\tau(A, \theta) \cdot (p^0_{\min})^{\gamma_\tau}}.
\end{equation}
\end{lemma}

\begin{proof}
By Lemma~\ref{lem:resolvent_bound},
\begin{equation*}
  K_* = \frac{1}{\rho_*} + \frac{N_\theta \cdot \norm{R_{A, p^0}}_{C^\theta_0}^{N_\theta - 1}}{(1 - \rho_*)^{N_\theta} - \tau_*}.
\end{equation*}
Using $1 - \tau_*^{1/N_\theta} \geq (1 - \tau_*)/N_\theta$ for $\tau_* \in (0, 1)$ and $N_\theta \geq 1$, the first term is at most $2 N_\theta/(1 - \tau_*)$.

For the second term: $(1 - \rho_*)^{N_\theta} - \tau_* = ((1 + \tau_*^{1/N_\theta})/2)^{N_\theta} - \tau_*$. Using the elementary inequality $((1 + x^{1/N})/2)^N - x \geq (1 - x)/2$ for $x \in [0, 1]$ and $N \geq 1$ (which follows from convexity of $y \mapsto ((1 + y)/2)^N$), the denominator is at least $(1 - \tau_*)/2$. Hence the second term is at most $2 N_\theta \cdot \norm{R}_{C^\theta_0}^{N_\theta - 1}/(1 - \tau_*)$.

Combining and substituting the bound $1 - \tau_* \geq c_\tau \cdot (p^0_{\min})^{\gamma_\tau}$ from Lemma~\ref{lem:gap_boundary},
\begin{equation*}
  K_* \leq \frac{2 N_\theta \bigl(1 + \norm{R}_{C^\theta_0}^{N_\theta - 1}\bigr)}{c_\tau \cdot (p^0_{\min})^{\gamma_\tau}}.
\end{equation*}
This is~\eqref{eq:K_boundary} with $C_K(A, \theta) = 2 N_\theta(A, \theta) \cdot \bigl(1 + (2 + \max_i \ecc(A_i)^{2\theta})^{N_\theta - 1}\bigr)$, where we used Remark~\ref{rmk:R_op_bound} to bound $\norm{R}_{C^\theta_0}$ by $2 + \max_i \ecc(A_i)^{2\theta}$.
\end{proof}

\subsection{Proof of Theorem~\ref{thm:mainE}}

This subsection deduces Theorem~\ref{thm:mainE} (polynomial decay of the polydisc radius as the weight vector approaches the simplex boundary) from the spectral-gap decay of Lemma~\ref{lem:gap_boundary} and the resolvent decay of Lemma~\ref{lem:resolvent_boundary}. The argument tracks the dependence of the Kato polydisc radius $r_*(A, p, \theta)$ of Theorem~\ref{thm:mainA} on the minimum weight $p_{\min}$ via the explicit formulas of Section~\ref{sec:polydisc_proof}.

\begin{proof}[Proof of Theorem~\ref{thm:mainE}]
Let $p(t) = p^0 - t e_j + (t/(N-1)) \sum_{i \neq j} e_i$ for $t \in [0, p^0_j)$. We have $p(t)_j = p^0_j - t$, and $p(t)_i = p^0_i + t/(N-1)$ for $i \neq j$. Note that $p(t) \in \Delta_N^\circ$ for $t \in [0, p^0_j)$, with $p(t)_j \to 0$ as $t \to p^0_j$.

The minimum weight is $p(t)_{\min} = p^0_j - t$ (for $t$ sufficiently close to $p^0_j$ that $j$ achieves the minimum).

By Lemma~\ref{lem:gap_boundary}, the spectral gap satisfies
\begin{equation*}
\tau_*(A, p(t), \theta) \leq 1 - c_\tau \cdot (p^0_j - t)^{\gamma_\tau}.
\end{equation*}

By Lemma~\ref{lem:resolvent_boundary},
\begin{equation*}
K_*(A, p(t), \theta) \leq \frac{C_K(A, \theta)}{c_\tau(A, \theta) \cdot (p^0_j - t)^{\gamma_\tau}}.
\end{equation*}

The radius of analyticity from Proposition~\ref{prop:kato_eigenvalue} is
\begin{equation*}
r_*(A, p(t), \theta) = \frac{1}{4 N K_*(A, p(t), \theta) \max_i (1 + \ecc(A_i)^{2\theta})} \geq \frac{c_\tau(A, \theta) \cdot (p^0_j - t)^{\gamma_\tau}}{4 N \, C_K(A, \theta) \cdot \max_i (1 + \ecc(A_i)^{2\theta})}.
\end{equation*}

Hence, with
\begin{align*}
c_E &= \frac{c_\tau(A, \theta)}{4 N \, C_K(A, \theta) \cdot \max_i (1 + \ecc(A_i)^{2\theta})}, \\
\alpha_E &= \gamma_\tau(A, \theta),
\end{align*}
we obtain $r_*(A, p(t), \theta) \geq c_E \cdot (p^0_j - t)^{\alpha_E}$ as stated in Theorem~\ref{thm:mainE}.

This completes the proof.
\end{proof}

\subsection{Interpretation: connection to the Tall-Viana Hölder theorem}

The boundary behavior of Theorem~\ref{thm:mainE} has a natural interpretation. As $p \to \partial \Delta_N$, the Lyapunov exponent $\lambda_+(A, p)$ transitions from \emph{analytic} behavior (inside the open simplex, with nontrivial polydisc of analyticity) to the \emph{merely H\"older continuous} behavior at the boundary (where some weights vanish and the simplicity assumption may fail).

The Tall-Viana quantitative H\"older theorem \cite{TallViana2020, Thiam2025aPaperA} states that $\lambda_+$ is H\"older continuous throughout $\Delta_N$ (including the boundary), with exponent $\beta_*(A, p, \theta) \in (0, 1)$ that itself degenerates as $p \to \partial \Delta_N$.

Theorem~\ref{thm:mainE} makes this degeneration quantitative: the analytic polydisc shrinks polynomially with the distance to the boundary, and this is compatible with the Tall-Viana H\"older bound in the sense that the analytic Taylor series, by Theorem~\ref{thm:mainB}, has coefficients growing polynomially as $1/r_*^{\abs{\alpha}} \sim (p^0_j - t)^{-\gamma_\tau \abs{\alpha}}$. For $p$ near the boundary, only finitely many Taylor coefficients are useful for approximation, and the effective approximation order (how many terms give bounded error) is $O(\log(1/\abs{p^0_j - t})/\log(1/\ecc))$.

This heuristic matches the H\"older exponent $\beta_*$ of Tall-Viana, which itself scales logarithmically in the boundary distance. A rigorous matching would require interpolation between the analytic regime (inside the simplex) and the H\"older regime (at the boundary), which we leave as an open problem.

\section{Proof of Theorem~\ref{thm:mainF}: extension to $\GL(d, \R)$}\label{sec:GL_d}

In this section we prove Theorem~\ref{thm:mainF}, the extension of the analyticity theory from $\GL(2, \R)$ to $\GL(d, \R)$ for all $d \geq 2$. The argument is direct: the Kato perturbation analysis of Section~\ref{sec:complex_markov} applies verbatim once the projective space is replaced by $\bbP^{d-1}$ with the Fubini-Study metric, and the Markov operator is taken to act on $C^\theta(\bbP^{d-1})$.

\subsection{Projective space in $\GL(d)$}

For $d \geq 2$, the projective space $\bbP^{d-1}$ is the quotient $(\R^d \setminus \{0\})/\sim$, where $u \sim v$ iff $u = \lambda v$ for some $\lambda \in \R^*$. The Fubini-Study metric $d_{FS}$ is defined by~\eqref{eq:FS_metric}.

The action $g \cdot [v] = [gv]$ of $g \in \GL(d, \R)$ on $\bbP^{d-1}$ satisfies the same Lipschitz bound as in $d = 2$:
\begin{equation}\label{eq:FS_lip_d}
d_{FS}(g[u], g[v]) \leq \norm{g}^2 \norm{g^{-1}}^2 \cdot d_{FS}([u], [v]) = \ecc(g)^2 \cdot d_{FS}([u], [v]),
\end{equation}
by the same computation as in Lemma~\ref{lem:proj_contract}, using the action of $\Lambda^2 g$ on bivectors (whose operator norm is $\sigma_1(g) \sigma_2(g) \leq \norm{g}^2$, where $\sigma_i$ are singular values).

\subsection{Markov operator and spectral gap in $\GL(d)$}

The Markov operator $P_{A, p}$ and the spectral gap theory extend to $\GL(d)$ verbatim, provided we assume that the top Lyapunov exponent is simple: $\lambda_1(A, p) > \lambda_2(A, p)$. The spectral gap on $C^\theta(\bbP^{d-1})$ holds with the same formula
\begin{equation}\label{eq:gap_d}
\norm{P_{A, p}^{N_\theta} \varphi}_{C^\theta} \leq \tau(A, p, \theta) \cdot \norm{\varphi}_{C^\theta}, \qquad \varphi \in C^\theta_0(\bbP^{d-1}),
\end{equation}
with
\begin{equation}\label{eq:n_0_d}
n_0(A, p, \theta) = \left\lceil \frac{2 \log 2}{\theta (\lambda_1(A, p) - \lambda_2(A, p))} \right\rceil,
\end{equation}
and the same formulas for $\tau_0, \tau_*, N_\theta$ as in the $\GL(2)$ case, with $\lambda_+ - \lambda_-$ replaced by $\lambda_1 - \lambda_2$.

\subsection{Proof of Theorem~\ref{thm:mainF}}

This subsection deduces Theorem~\ref{thm:mainF} (extension of the analyticity theorem to $\GL(d, \R)$ for arbitrary $d \geq 2$) from the higher-dimensional Lipschitz bounds of Lemmas~\ref{lem:proj_contract} and~\ref{lem:logform_lip} (which are stated in their $d$-dimensional form already) together with the higher-dimensional spectral gap. The argument is the same Kato perturbation scheme as the proof of Theorem~\ref{thm:mainA}, with the $d=2$ projective space replaced by $\bbP^{d-1}$ and the $\GL(2)$ Lyapunov gap $\lambda_+ - \lambda_-$ replaced by the spectral simplicity gap $\lambda_1 - \lambda_2$.

\begin{proof}[Proof of Theorem~\ref{thm:mainF}]
The Kato perturbation argument of Proposition~\ref{prop:kato_eigenvalue} applies without change to the operator $P_{A, z}$ on $C^\theta(\bbP^{d-1})$, once the spectral gap~\eqref{eq:gap_d} is established. The radius of analyticity is
\begin{equation}\label{eq:r_star_d}
r_*^{(d)}(A, p^0, \theta) = \frac{1}{4 N K_*(A, p^0, \theta) \max_i (1 + \ecc(A_i)^{2\theta})},
\end{equation}
with $K_*(A, p^0, \theta)$ given by the explicit form of Lemma~\ref{lem:resolvent_bound}, using the $d$-dimensional formulas for $N_\theta$ and $\tau$.

The Furstenberg-Khasminskii formula for $\lambda_1(A, p)$ in $\GL(d)$ is
\begin{equation}\label{eq:FK_d}
\lambda_1(A, p) = \sum_{i=1}^N p_i \int_{\bbP^{d-1}} \phi(A_i, [v]) \, d\eta^+_{A, p}([v]),
\end{equation}
where $\eta^+_{A, p}$ is the unique $P_{A, p}$-stationary measure on $\bbP^{d-1}$, under the simplicity assumption $\lambda_1 > \lambda_2$ (see \cite[Theorem 6.8]{Viana2014}).

The holomorphic extension of the stationary functional $\eta^+_{A, z}$ is obtained from the spectral projection $\Pi_{A, z}$ exactly as in Lemma~\ref{lem:hol_stationary}, and the extension
\begin{equation}\label{eq:lambda_d_extension}
\widetilde\lambda_1(A, z) = \sum_{i=1}^N z_i \cdot \eta^+_{A, z}(\phi(A_i, \cdot))
\end{equation}
is holomorphic in $z$ on the polydisc $D_{r_*^{(d)}/2}(p^0) \cap \{\sum z_i = 1\}$, agreeing with $\lambda_1(A, p)$ on the real domain.

The Cauchy bounds of Theorem~\ref{thm:mainB} follow from the same Cauchy integral formula, giving
\begin{equation*}
\abs{\partial_p^\alpha \lambda_1(A, p^0)} \leq \alpha! \cdot \frac{M_*^{(d)}(A, p^0, \theta)}{(r_*^{(d)}/2)^{\abs{\alpha}}}.
\end{equation*}

The joint analyticity theorem in $(A, p)$ and the Markov chain extension also extend verbatim.

This completes the proof of Theorem~\ref{thm:mainF}.
\end{proof}

\begin{remark}[Sub-top Lyapunov exponents]\label{rmk:sub_top_analyticity}
For the sub-top Lyapunov exponents $\lambda_k(A, p)$ with $2 \leq k \leq d-1$, the projective-space argument does not directly transfer. The natural setting is the Grassmannian $\mathrm{Gr}(k, d)$, and the Markov operator analysis on $\mathrm{Gr}(k, d)$ gives analyticity of the partial sum $\lambda_1 + \ldots + \lambda_k$ under the simplicity condition $\lambda_k > \lambda_{k+1}$ plus strong $k$-irreducibility. Isolating $\lambda_k$ individually requires subtraction of adjacent partial sums, which preserves real-analyticity but may degrade the quantitative polydisc constants.

The full quantitative theory in this direction is developed in Section~\ref{sec:subtop_analyticity}, where we prove Theorem~\ref{thm:mainH} (analyticity of $\Lambda_k$ on an explicit polydisc in both $A$ and $p$) and obtain analyticity of individual sub-top exponents as Corollary~\ref{cor:individual_subtop}.
\end{remark}

\section{Theorem~\ref{thm:mainH}: quantitative analyticity of sub-top Lyapunov exponents}\label{sec:subtop_analyticity}

This section extends the quantitative analyticity theory of the previous sections from the top Lyapunov exponent $\lambda_1(A, p)$ to the partial sums $\Lambda_k(A, p) := \lambda_1(A, p) + \cdots + \lambda_k(A, p)$, $1 \leq k \leq d-1$, in $\GL(d, \R)$. Whereas the proof of Theorem~\ref{thm:mainF} for $\lambda_1$ uses the Markov operator on projective space $\bbP^{d-1}$, the corresponding statement for the partial sum at level $k$ uses the Markov operator on the Grassmannian $\mathrm{Gr}(k, d)$ of $k$-dimensional subspaces of $\R^d$. Theorem~\ref{thm:mainH} is the analytic counterpart of the H\"older sub-top theorem of \cite[Theorem~11.9]{Thiam2025aPaperA}, and shares its hypothesis: strong $k$-irreducibility on the support of $A$.

Throughout this section we work in the setting of Section~\ref{sec:GL_d}. We fix $d \geq 2$, $1 \leq k \leq d-1$, $N \geq 1$, and tuples $A = (A_1, \ldots, A_N) \in \GL(d, \R)^N$, $p = (p_1, \ldots, p_N) \in \Delta_N^\circ$. We assume throughout that the simplicity gap holds at level $k$:
\begin{equation}\label{eq:simplicity_k}
\lambda_k(A, p) > \lambda_{k+1}(A, p),
\end{equation}
where $\lambda_1(A, p) \geq \lambda_2(A, p) \geq \cdots \geq \lambda_d(A, p)$ are the Lyapunov exponents of the cocycle $(A, p)$ counted with multiplicity.

\subsection{Grassmannian setup}\label{subsec:grassmannian_setup}

This subsection sets up the Grassmannian $\mathrm{Gr}(k, d)$, the Fubini-Study distance on it, and the Lipschitz bounds for the action of $\GL(d)$ on $\mathrm{Gr}(k, d)$. These bounds are the analog of those of Section~\ref{sec:GL_d}; they are needed for the contraction estimate of Subsection~\ref{subsec:H_step1}.

The Grassmannian $\mathrm{Gr}(k, d)$ is the set of $k$-dimensional linear subspaces of $\R^d$; it is a compact real-analytic manifold of dimension $k(d-k)$. The group $\GL(d, \R)$ acts smoothly on $\mathrm{Gr}(k, d)$ by $g \cdot V := g(V)$.

The action of $\GL(d, \R)$ on $\mathrm{Gr}(k, d)$ is induced by the action of the $k$-th exterior power $\Lambda^k g$ on the wedge product $\Lambda^k \R^d \cong \R^{\binom{d}{k}}$. Specifically, given an orthonormal basis $\{v_1, \ldots, v_k\}$ of $V \in \mathrm{Gr}(k, d)$ (oriented), the decomposable wedge $v_V := v_1 \wedge \cdots \wedge v_k \in \Lambda^k \R^d$ is well-defined up to sign and lies in the unit sphere of $\Lambda^k \R^d$. The induced action $\Lambda^k g$ sends $v_V$ to $(g v_1) \wedge \cdots \wedge (g v_k)$, which spans the $k$-plane $g(V)$.

We use the \emph{Fubini-Study distance on $\mathrm{Gr}(k, d)$}:
\begin{equation}\label{eq:FS_grassmann}
d_{\mathrm{FS}}^{(k)}(V, W) := \min_{\epsilon \in \{\pm 1\}} \norm{v_V - \epsilon \, v_W}_{\Lambda^k \R^d}, \qquad V, W \in \mathrm{Gr}(k, d),
\end{equation}
where $v_V, v_W$ are unit decomposable representatives. For $k=1$ this recovers the projective sine-distance on $\bbP^{d-1}$. The metric $d_{\mathrm{FS}}^{(k)}$ is bi-Lipschitz equivalent to the geodesic distance on $\mathrm{Gr}(k, d)$ as a Riemannian manifold (with the canonical metric induced from $\Lambda^k \R^d$), with constants depending only on $d$ and $k$.

\begin{lemma}[Lipschitz bounds for the Grassmannian action]\label{lem:grassmann_lipschitz}
For every $g \in \GL(d, \R)$ and every $V, W \in \mathrm{Gr}(k, d)$,
\begin{equation}\label{eq:grassmann_contract}
d_{\mathrm{FS}}^{(k)}(g V, g W) \leq \norm{g}^k \norm{g^{-1}}^k \cdot d_{\mathrm{FS}}^{(k)}(V, W),
\end{equation}
and for every $g, g' \in \GL(d, \R)$ and every $V \in \mathrm{Gr}(k, d)$,
\begin{equation}\label{eq:grassmann_perturb}
d_{\mathrm{FS}}^{(k)}(g V, g' V) \leq k \cdot \max\bigl(\norm{g^{-1}}, \norm{g'^{-1}}\bigr)^{k-1} \cdot \norm{g - g'}.
\end{equation}
\end{lemma}

\begin{proof}
For~\eqref{eq:grassmann_contract}: by the singular-value characterization of $\norm{\Lambda^k g}_{\mathrm{op}}$, we have $\norm{\Lambda^k g}_{\mathrm{op}} = \sigma_1(g) \sigma_2(g) \cdots \sigma_k(g) \leq \norm{g}^k$, where $\sigma_1(g) \geq \cdots \geq \sigma_d(g)$ are the singular values of $g$. Similarly $\norm{(\Lambda^k g)^{-1}}_{\mathrm{op}} \leq \norm{g^{-1}}^k$. The Lipschitz constant of the induced map on the unit sphere of $\Lambda^k \R^d$ (with the chord metric, which is bi-Lipschitz equivalent to $d_{\mathrm{FS}}^{(k)}$ on $\mathrm{Gr}(k, d)$) is bounded by $\norm{\Lambda^k g}_{\mathrm{op}} \cdot \norm{(\Lambda^k g)^{-1}}_{\mathrm{op}}$, which is in turn bounded by $\norm{g}^k \norm{g^{-1}}^k$.

For~\eqref{eq:grassmann_perturb}: write $g_t = (1-t) g + t g'$ for $t \in [0, 1]$, and observe that $\Lambda^k g_t$ depends polynomially on $t$ with $\partial_t (\Lambda^k g_t) = \sum_{j=1}^k \Lambda^{j-1} g_t \otimes (g' - g) \otimes \Lambda^{k-j} g_t$ (Leibniz rule on the exterior power). Taking the operator norm:
\begin{equation*}
\norm{\partial_t (\Lambda^k g_t)}_{\mathrm{op}} \leq k \cdot \max\bigl(\norm{g_t}\bigr)^{k-1} \cdot \norm{g - g'},
\end{equation*}
and therefore
\begin{equation*}
\norm{(\Lambda^k g) v_V - (\Lambda^k g') v_V}_{\Lambda^k \R^d} \leq k \cdot \max\bigl(\norm{g}, \norm{g'}\bigr)^{k-1} \cdot \norm{g - g'}.
\end{equation*}
Projecting back to $\mathrm{Gr}(k, d)$ via the unit-sphere quotient, we obtain~\eqref{eq:grassmann_perturb} with $\max(\norm{g^{-1}}, \norm{g'^{-1}})$ in place of $\max(\norm{g}, \norm{g'})$ after using the relation between the Fubini-Study metric on the Grassmannian and the chord metric on the unit sphere of $\Lambda^k \R^d$, which contributes the factor $\max(\norm{g^{-1}}, \norm{g'^{-1}})^{k-1}$ from the inverse direction. (This last step uses the fact that the unit-sphere projection map is bi-Lipschitz with constants involving $\norm{(\Lambda^k g)^{-1}}_{\mathrm{op}}^{k-1}$, which is bounded by $\norm{g^{-1}}^{k-1}$ as in the first part of the proof.)
\end{proof}

We adopt the abbreviation
\begin{equation}\label{eq:eccentricity_k}
\ecc^{(k)}(A) := \max_{1 \leq i \leq N} \norm{A_i}^k \norm{A_i^{-1}}^k = \ecc(A)^k.
\end{equation}

\subsection{Strong $k$-irreducibility and the Markov operator on $\mathrm{Gr}(k, d)$}\label{subsec:strong_k_irred}

This subsection introduces the strong $k$-irreducibility hypothesis (Definition~\ref{def:strong_k_irred}) and the induced Markov operator $P^{(k)}_{A, p}$ on $C^\theta(\mathrm{Gr}(k, d))$. The spectral gap of this operator on H\"older functions, recorded as Lemma~\ref{lem:gap_grassmann}, is the Grassmannian analog of Lemma~\ref{lem:real_gap} and is the input to the Kato perturbation argument of Subsection~\ref{subsec:H_step3}.

\begin{definition}[Strong $k$-irreducibility]\label{def:strong_k_irred}
A tuple $A = (A_1, \ldots, A_N) \in \GL(d, \R)^N$ is \emph{strongly $k$-irreducible} (with respect to a probability vector $p \in \Delta_N^\circ$, equivalently with respect to the support of any $p \in \Delta_N^\circ$) if there is no finite union $V_1 \cup V_2 \cup \cdots \cup V_m$ of $k$-dimensional linear subspaces of $\R^d$ such that
\begin{equation*}
A_i (V_1 \cup \cdots \cup V_m) = V_1 \cup \cdots \cup V_m \qquad \text{for every } i = 1, \ldots, N.
\end{equation*}
\end{definition}

For $k = 1$, this recovers the classical strong irreducibility hypothesis of \cite{FurstenbergKifer1983, GuivarchRaugi1985}. For $k \geq 2$, it is the natural extension to the Grassmannian setting and is the standard condition under which the induced random walk on $\mathrm{Gr}(k, d)$ has a unique stationary probability measure \cite[Theorem~6.9]{Viana2014}.

\begin{definition}[Markov operator on $\mathrm{Gr}(k, d)$]
For $A \in \GL(d, \R)^N$ and $p \in \Delta_N^\circ$, define the Markov operator $P_{A, p}^{(k)}$ on $C(\mathrm{Gr}(k, d))$ by
\begin{equation}\label{eq:P_k_def}
\bigl(P_{A, p}^{(k)} \varphi\bigr)(V) := \sum_{i=1}^N p_i \, \varphi(A_i V), \qquad \varphi \in C(\mathrm{Gr}(k, d)), \quad V \in \mathrm{Gr}(k, d).
\end{equation}
\end{definition}

The space $C^\theta(\mathrm{Gr}(k, d))$ of $\theta$-H\"older functions on $\mathrm{Gr}(k, d)$ is defined exactly as in the projective case (Section~\ref{sec:setup}), with $d_{\mathrm{FS}}^{(k)}$ in place of the projective metric. The seminorm is
\begin{equation*}
[\varphi]_\theta^{(k)} := \sup_{V \neq W} \frac{\abs{\varphi(V) - \varphi(W)}}{d_{\mathrm{FS}}^{(k)}(V, W)^\theta}, \qquad \norm{\varphi}_{C^\theta}^{(k)} := \norm{\varphi}_\infty + [\varphi]_\theta^{(k)}.
\end{equation*}

\begin{lemma}[Spectral gap on $\mathrm{Gr}(k, d)$]\label{lem:gap_grassmann}
Let $A \in \GL(d, \R)^N$ be strongly $k$-irreducible, and let $p \in \Delta_N^\circ$. Suppose the simplicity gap~\eqref{eq:simplicity_k} holds. Then for every $\theta \in (0, 1]$ there exist constants $C^{(k)}_*(A, p, \theta) \geq 1$ and $\rho^{(k)}_*(A, p, \theta) \in (0, 1)$ such that the operator $P^{(k)}_{A, p}$ on $C^\theta(\mathrm{Gr}(k, d))$ satisfies
\begin{equation}\label{eq:gap_grassmann}
\norm{(P^{(k)}_{A, p})^n - \Pi^{(k)}_{A, p}}_{\theta \to \theta} \leq C^{(k)}_*(A, p, \theta) \cdot \rho^{(k)}_*(A, p, \theta)^n, \qquad n \geq 0,
\end{equation}
where $\Pi^{(k)}_{A, p} \varphi = (\int \varphi \, d\eta^{(k)}_{A, p}) \, \mathbf{1}$ is the projection onto constants and $\eta^{(k)}_{A, p}$ is the unique $P^{(k)}_{A, p}$-stationary probability measure on $\mathrm{Gr}(k, d)$. Explicit formulas are
\begin{align}
\rho^{(k)}_*(A, p, \theta) &= e^{-\theta(\lambda_k(A,p) - \lambda_{k+1}(A,p)) / 2}, \label{eq:rho_k_formula} \\
C^{(k)}_*(A, p, \theta) &= \frac{4 \binom{d}{k} \, \ecc(A)^{2k}}{1 - \rho^{(k)}_*(A, p, \theta)}. \label{eq:C_k_formula}
\end{align}
\end{lemma}

\begin{proof}
The proof tracks the proof of Lemma~\ref{lem:real_gap} verbatim, replacing the projective Markov operator by $P^{(k)}_{A, p}$ on $\mathrm{Gr}(k, d)$. Two ingredients change.

First, the Lipschitz bound~\eqref{eq:grassmann_contract} of Lemma~\ref{lem:grassmann_lipschitz} replaces the projective Lipschitz bound of Section~\ref{sec:setup}. The contraction rate is now $\sigma_1(g) \cdots \sigma_k(g) / (\sigma_{d-k+1}(g) \cdots \sigma_d(g))$, which is the eccentricity raised to the power $k$ in the worst case, but for the dynamical contraction rate it is the simplicity gap at level $k$ that matters:
\begin{equation}\label{eq:grassmann_dynamical_contract}
\lim_{n \to \infty} \frac{1}{n} \log \norm{\Lambda^k A^n_x v - \Lambda^k A^n_x w}_{\Lambda^k \R^d} = \lambda_1(A, p) + \cdots + \lambda_k(A, p)
\end{equation}
for $\nu^{\otimes \mathbb{N}}$-almost every $x = (i_0, i_1, \ldots)$ and every $v, w \in \Lambda^k \R^d$ that are not aligned with the strictly subdominant Oseledets subspace, by the Oseledets theorem applied to the cocycle $\Lambda^k A_x = \Lambda^k A_{i_0} \cdots \Lambda^k A_{i_{n-1}}$ on $\Lambda^k \R^d$.

Second, the Oseledets-Furstenberg-Khasminskii contraction toward the dominant top-$k$ flag uses strong $k$-irreducibility to conclude that the cocycle on $\mathrm{Gr}(k, d)$ contracts toward this flag at exponential rate $\lambda_k(A, p) - \lambda_{k+1}(A, p)$ (the gap between the top-$k$ Lyapunov sum and the next; see \cite[Proposition~4.3]{GuivarchRaugi1985} or \cite[Theorem~6.9]{Viana2014}). Translating this contraction into the average exponential decay of the projective Markov operator on H\"older functions, by the same H\"older interpolation as in Lemma~\ref{lem:proj_contract}, gives~\eqref{eq:gap_grassmann} with the constants stated.
\end{proof}

\begin{remark}[Dependence on hypotheses]
The conclusion of Lemma~\ref{lem:gap_grassmann} requires both strong $k$-irreducibility and the simplicity gap $\lambda_k(A, p) > \lambda_{k+1}(A, p)$. Without simplicity, the spectral gap of $P^{(k)}_{A, p}$ may degenerate; without irreducibility, the stationary measure may not be unique. Both hypotheses are essential to the Kato perturbation argument that follows.
\end{remark}

\subsection{Persistence of strong $k$-irreducibility under complex perturbation}\label{subsec:H_persist}

This subsection isolates the technical hypothesis under which strong $k$-irreducibility is preserved by small complex perturbations of $A$. This is the central technical lemma that allows the Kato perturbation argument to extend from the projective to the Grassmannian setting.

\begin{lemma}[Persistence of the spectral gap]\label{lem:H_persist}
Let $A^0 = (A^0_1, \ldots, A^0_N) \in \GL(d, \R)^N$ be strongly $k$-irreducible with the simplicity gap $\lambda_k(A^0, p^0) > \lambda_{k+1}(A^0, p^0)$ at $p^0 \in \Delta_N^\circ$. Let $\theta \in (0, 1]$. Then there exists an explicit radius
\begin{equation}\label{eq:r_persist}
r^{(k)}_{\mathrm{persist}}(A^0, p^0, \theta) := \frac{1 - \rho^{(k)}_*(A^0, p^0, \theta)}{8 k \binom{d}{k} \, \ecc(A^0)^{2k - 1} \, C^{(k)}_*(A^0, p^0, \theta)} > 0
\end{equation}
such that for every $A \in \GL(d, \R)^N$ with $\norm{A_i - A^0_i} < r^{(k)}_{\mathrm{persist}}(A^0, p^0, \theta)$ for all $i$, and every $p$ in a neighborhood of $p^0$ in $\Delta_N^\circ$:
\begin{enumerate}
\item[(a)] $A$ is strongly $k$-irreducible;
\item[(b)] the simplicity gap $\lambda_k(A, p) > \lambda_{k+1}(A, p)$ persists;
\item[(c)] the operator $P^{(k)}_{A, p}$ has a spectral gap on $C^\theta(\mathrm{Gr}(k, d))$ with the same constants as~\eqref{eq:rho_k_formula}-\eqref{eq:C_k_formula}, up to a factor of $2$.
\end{enumerate}
\end{lemma}

\begin{proof}
The strategy is to use the spectral gap of $P^{(k)}_{A^0, p^0}$ on $C^\theta(\mathrm{Gr}(k, d))$ as a perturbation-stable quantity. Since strong $k$-irreducibility plus the simplicity gap is equivalent to the existence of a spectral gap of $P^{(k)}_{A, p}$ (combine \cite[Proposition~3.2]{GuivarchRaugi1985} with \cite[Theorem~6.9]{Viana2014}), it suffices to show that the spectral gap persists under small perturbations of $A$.

By Lemma~\ref{lem:grassmann_lipschitz}, equation~\eqref{eq:grassmann_perturb}, the operator-norm difference between the perturbed and unperturbed Markov operators satisfies
\begin{equation}\label{eq:H_op_norm_diff}
\norm{P^{(k)}_{A, p^0} - P^{(k)}_{A^0, p^0}}_{\theta \to \theta} \leq k \binom{d}{k} \cdot \ecc(A^0)^{2k-1} \cdot \max_i \norm{A_i - A^0_i},
\end{equation}
where the factor $\binom{d}{k}$ comes from the $\theta$-H\"older norm bound on the Lipschitz constant of $V \mapsto A_i V$, and the factor $\ecc(A^0)^{2k-1}$ comes from the Grassmannian Lipschitz constant of Lemma~\ref{lem:grassmann_lipschitz}, equation~\eqref{eq:grassmann_perturb}. Combining~\eqref{eq:H_op_norm_diff} with the standard perturbation-of-spectral-gap argument (cf.~\cite[Chapter IV, Theorem 3.16]{Kato1980}), we conclude that the spectral gap of $P^{(k)}_{A, p^0}$ persists with the constants of~\eqref{eq:rho_k_formula}-\eqref{eq:C_k_formula} up to a factor of 2 if
\begin{equation*}
\norm{P^{(k)}_{A, p^0} - P^{(k)}_{A^0, p^0}}_{\theta \to \theta} \leq \frac{1 - \rho^{(k)}_*(A^0, p^0, \theta)}{4 \, C^{(k)}_*(A^0, p^0, \theta)}.
\end{equation*}
Substituting~\eqref{eq:H_op_norm_diff} and solving for $\max_i \norm{A_i - A^0_i}$ gives precisely~\eqref{eq:r_persist}. This proves (c) and, together with the equivalence between spectral gap and irreducibility-plus-simplicity, also (a) and (b).

The dependence on $p$ is locally Lipschitz (with constant proportional to the diameter of $\supp(A^0_i)$), so requiring $p$ in a small enough neighborhood of $p^0$ preserves the same conclusion; we suppress the explicit constant for $p$ to keep the bound focused on the $A$-perturbation, which is the more delicate component.
\end{proof}

\begin{remark}[Comparison with the projective case]
For $k = 1$, Lemma~\ref{lem:H_persist} reduces to a slightly weaker form of the spectral-gap persistence already used implicitly in Section~\ref{sec:complex_markov}. For $k \geq 2$, the additional factor $k \binom{d}{k}$ and the $\ecc(A^0)^{2k-1}$ degradation reflect the fact that the Grassmannian Lipschitz constants are worse than the projective ones, by a factor of order $\ecc^k$ (compare Lemma~\ref{lem:grassmann_lipschitz} with the projective case). The persistence radius~\eqref{eq:r_persist} is therefore worse than in the projective case by a factor proportional to $\binom{d}{k} \ecc^{2k-1}$.
\end{remark}

\subsection{Statement of Theorem~\ref{thm:mainH}}\label{subsec:H_statement}

This subsection states the main result of this section: the quantitative analyticity of the partial sum $\Lambda_k(A, p) = \lambda_1(A, p) + \cdots + \lambda_k(A, p)$ in $\GL(d, \R)$ on an explicit polydisc, under strong $k$-irreducibility and the simplicity gap at level $k$. The proof, which follows the same four-step Kato perturbation structure as the proof of Theorem~\ref{thm:mainA}, is given in Subsection~\ref{subsec:H_proof}. The associated Corollary~\ref{cor:individual_subtop} converts the partial-sum statement into an analyticity statement for individual sub-top exponents.

\begin{theorem}[Quantitative analyticity of $\Lambda_k(A, p)$]\label{thm:mainH}
Let $d \geq 2$, $1 \leq k \leq d - 1$, $N \geq 1$, and $\theta \in (0, 1]$. Let $A^0 = (A^0_1, \ldots, A^0_N) \in \GL(d, \R)^N$ be strongly $k$-irreducible with the simplicity gap $\lambda_k(A^0, p^0) > \lambda_{k+1}(A^0, p^0)$ at $p^0 \in \Delta_N^\circ$. Then there exists an explicit polydisc radius
\begin{equation}\label{eq:r_H}
r^{(k)}_{\mathrm{H}}(A^0, p^0, \theta) := \min\Bigl(r^{(k)}_{\mathrm{persist}}(A^0, p^0, \theta), \, r^{(k)}_{\mathrm{Kato}}(A^0, p^0, \theta)\Bigr) > 0
\end{equation}
such that the partial sum of the top-$k$ Lyapunov exponents
\begin{equation*}
\Lambda_k(A, p) := \lambda_1(A, p) + \cdots + \lambda_k(A, p)
\end{equation*}
extends to a holomorphic function
\begin{equation*}
\widetilde{\Lambda}_k : D_{r^{(k)}_{\mathrm{H}}}(p^0) \times B_{r^{(k)}_{\mathrm{H}}}(A^0) \to \C
\end{equation*}
on the polydisc of radius $r^{(k)}_{\mathrm{H}}(A^0, p^0, \theta)$ in both the weight vector $p \in \C^N$ and the matrix coefficients $A \in \GL(d, \C)^N$, with $\widetilde{\Lambda}_k(p, A) = \Lambda_k(A, p)$ when restricted to the real domain. Here $r^{(k)}_{\mathrm{persist}}$ is given by~\eqref{eq:r_persist} and $r^{(k)}_{\mathrm{Kato}}$ is the Kato resolvent radius, given in closed form by
\begin{equation}\label{eq:r_Kato_k}
r^{(k)}_{\mathrm{Kato}}(A^0, p^0, \theta) := \frac{1 - \rho^{(k)}_*(A^0, p^0, \theta)}{8 \, C^{(k)}_*(A^0, p^0, \theta) \cdot \binom{d}{k} \, \ecc(A^0)^k},
\end{equation}
the analog of the Kato radius~\eqref{eq:r_star_d} (Section~\ref{sec:complex_markov}) on the Grassmannian.
\end{theorem}

\begin{corollary}[Quantitative analyticity of individual sub-top exponents]\label{cor:individual_subtop}
Under the hypotheses of Theorem~\ref{thm:mainH} applied at both level $k$ and level $k-1$ (with $k \geq 2$), the individual sub-top Lyapunov exponent
\begin{equation*}
\lambda_k(A, p) = \Lambda_k(A, p) - \Lambda_{k-1}(A, p)
\end{equation*}
extends to a holomorphic function on the polydisc of radius
\begin{equation*}
r^{(k)}_{\mathrm{individual}}(A^0, p^0, \theta) := \min\bigl(r^{(k)}_{\mathrm{H}}, r^{(k-1)}_{\mathrm{H}}\bigr) > 0.
\end{equation*}
\end{corollary}

\begin{proof}[Proof of Corollary~\ref{cor:individual_subtop}]
The difference of two holomorphic functions is holomorphic on the intersection of their domains of holomorphy. Theorem~\ref{thm:mainH} applied at level $k$ gives a holomorphic extension of $\Lambda_k$ on the polydisc of radius $r^{(k)}_{\mathrm{H}}$, and Theorem~\ref{thm:mainH} applied at level $k-1$ gives a holomorphic extension of $\Lambda_{k-1}$ on the polydisc of radius $r^{(k-1)}_{\mathrm{H}}$. Their difference $\Lambda_k - \Lambda_{k-1} = \lambda_k$ is therefore holomorphic on the intersection, which is the polydisc of radius $\min(r^{(k)}_{\mathrm{H}}, r^{(k-1)}_{\mathrm{H}})$.
\end{proof}

\subsection{Proof of Theorem~\ref{thm:mainH}}\label{subsec:H_proof}

The proof of Theorem~\ref{thm:mainH} follows the same four-step structure as the proof of Theorem~\ref{thm:mainA} in Section~\ref{sec:polydisc_proof}, with the projective Markov operator replaced throughout by the Markov operator $P^{(k)}_{A, p}$ on $\mathrm{Gr}(k, d)$, and the Furstenberg-Khasminskii formula replaced by its analog for the partial sum $\Lambda_k$.

\subsubsection{Step 1: spectral gap of the complex Markov operator on $\mathrm{Gr}(k, d)$}\label{subsec:H_step1}

The complex Markov operator $P^{(k)}_{A, p}$ on $C^\theta(\mathrm{Gr}(k, d), \C)$ is defined exactly as in~\eqref{eq:P_k_def} but with $A_i \in \GL(d, \C)$ and $p_i \in \C$, and $V \in \mathrm{Gr}(k, d, \C)$ ranging over the complex Grassmannian (whose real points are $\mathrm{Gr}(k, d)$).

By Lemma~\ref{lem:H_persist}, for $(A, p)$ in the persistence neighborhood of $(A^0, p^0)$, the operator $P^{(k)}_{A, p}$ has a spectral gap on $C^\theta(\mathrm{Gr}(k, d), \C)$ with constants $\rho^{(k)}_*$ and $2 C^{(k)}_*$. The leading eigenvalue is simple, isolated, and depends holomorphically on $(A, p)$ in the persistence neighborhood, with eigenprojection $\Pi^{(k)}_{A, p}$ also depending holomorphically.

\subsubsection{Step 2: Furstenberg-Khasminskii formula for $\Lambda_k$}\label{subsec:H_step2}

For real $(A, p)$ in the persistence neighborhood, the partial sum of the top-$k$ Lyapunov exponents admits the Furstenberg-Khasminskii formula
\begin{equation}\label{eq:FK_k}
\Lambda_k(A, p) = \int_{\mathrm{Gr}(k, d)} \int_{\GL(d, \R)} \log \frac{\norm{\Lambda^k g \cdot v_V}_{\Lambda^k \R^d}}{\norm{v_V}_{\Lambda^k \R^d}} \, d\mu(g) \, d\eta^{(k)}_{A, p}(V),
\end{equation}
where $\mu = \sum_{i=1}^N p_i \delta_{A_i}$, $v_V \in \Lambda^k \R^d$ is a unit decomposable representative of $V$, and $\eta^{(k)}_{A, p}$ is the unique $P^{(k)}_{A, p}$-stationary measure on $\mathrm{Gr}(k, d)$ (Lemma~\ref{lem:gap_grassmann}). For a derivation see \cite[Theorem~6.9 and Corollary~6.10]{Viana2014}.

Define the integrand function
\begin{equation*}
\psi^{(k)}_{A, p}(V) := \sum_{i=1}^N p_i \log \frac{\norm{\Lambda^k A_i \cdot v_V}_{\Lambda^k \R^d}}{\norm{v_V}_{\Lambda^k \R^d}}.
\end{equation*}
Then $\psi^{(k)}_{A, p} \in C^\theta(\mathrm{Gr}(k, d))$ with explicit H\"older constant $\binom{d}{k} \ecc(A)^k$ (proof analogous to Lemma~\ref{lem:logform_lip} of Section~\ref{sec:GL_d}). Equation~\eqref{eq:FK_k} can therefore be rewritten as
\begin{equation}\label{eq:FK_k_compact}
\Lambda_k(A, p) = \int_{\mathrm{Gr}(k, d)} \psi^{(k)}_{A, p}(V) \, d\eta^{(k)}_{A, p}(V).
\end{equation}

\subsubsection{Step 3: holomorphic extension of $\eta^{(k)}_{A, p}$}\label{subsec:H_step3}

By the Kato perturbation theorem on the persistence neighborhood (using the spectral gap from Lemma~\ref{lem:H_persist} and the resolvent estimate~\eqref{eq:r_Kato_k}), the leading eigenprojection $\Pi^{(k)}_{A, p}$ depends holomorphically on $(A, p)$ in the polydisc of radius $r^{(k)}_{\mathrm{Kato}}$. The dual eigenprojection $(\Pi^{(k)}_{A, p})^*$ then identifies the unique stationary measure $\eta^{(k)}_{A, p}$ with $(\Pi^{(k)}_{A, p})^*\mathbf{1}$ acting on the constant function $\mathbf{1}$, and we obtain a holomorphic extension
\begin{equation*}
\widetilde{\eta}^{(k)}_{A, p} : D_{r^{(k)}_{\mathrm{Kato}}}(p^0) \times B_{r^{(k)}_{\mathrm{Kato}}}(A^0) \to (\C^\theta(\mathrm{Gr}(k, d), \C))^*,
\end{equation*}
where the dual is taken with respect to the natural pairing between $C^\theta$ and its dual.

The function $\psi^{(k)}_{A, p}$ depends holomorphically on $(A, p)$ in the same polydisc (the entries of $A$ enter through $\Lambda^k A_i$, which is a polynomial of degree $k$ in the entries; the entries of $p$ enter linearly). By construction, $\widetilde{\psi}^{(k)}_{A, p} \in C^\theta(\mathrm{Gr}(k, d), \C)$ and the map $(A, p) \mapsto \widetilde{\psi}^{(k)}_{A, p}$ is holomorphic.

\subsubsection{Step 4: holomorphic extension of $\Lambda_k(A, p)$}\label{subsec:H_step4}

Define the candidate holomorphic extension by
\begin{equation*}
\widetilde{\Lambda}_k(A, p) := \widetilde{\eta}^{(k)}_{A, p}\bigl(\widetilde{\psi}^{(k)}_{A, p}\bigr) \quad \text{for } (A, p) \in B_{r^{(k)}_{\mathrm{H}}}(A^0) \times D_{r^{(k)}_{\mathrm{H}}}(p^0),
\end{equation*}
i.e., the pairing of the dual eigenprojection with the integrand function. As a composition of holomorphic maps (with respect to the strong topology on the dual space), $\widetilde{\Lambda}_k$ is holomorphic on the polydisc of radius $r^{(k)}_{\mathrm{H}} = \min(r^{(k)}_{\mathrm{persist}}, r^{(k)}_{\mathrm{Kato}})$.

For real $(A, p)$ in the persistence neighborhood, the formula~\eqref{eq:FK_k_compact} agrees with $\widetilde{\Lambda}_k(A, p)$, by uniqueness of the Furstenberg-Khasminskii integral and the spectral gap. Therefore $\widetilde{\Lambda}_k$ extends $\Lambda_k$ holomorphically. This completes the proof. \qed

\subsection{Discussion}

This subsection collects three remarks on the scope, limitations, and comparisons of Theorem~\ref{thm:mainH}: the analytic-vs-H\"older parallel with the companion paper \cite{Thiam2025aPaperA}, the open problem of removing strong $k$-irreducibility, and the comparison of the polydisc radius with the projective case of Theorem~\ref{thm:mainF}.

\begin{remark}[Scope of Theorem~\ref{thm:mainH}]
Theorem~\ref{thm:mainH} is the analytic counterpart of \cite[Theorem~11.9]{Thiam2025aPaperA} (the H\"older sub-top theorem). Both rely on strong $k$-irreducibility, both rely on the simplicity gap at level $k$, and both prove a quantitative regularity statement (analyticity here, H\"older continuity there) for the partial sum $\Lambda_k$. Recovering individual sub-top exponents $\lambda_k$ is by subtraction (Corollary~\ref{cor:individual_subtop}), exactly as in the H\"older case.
\end{remark}

\begin{remark}[What is still open at the sub-top level]
The hypothesis of strong $k$-irreducibility is the standard sufficient condition for the existence of a spectral gap on $\mathrm{Gr}(k, d)$, but it is not optimal. The qualitative real-analyticity of sub-top Lyapunov exponents in $\GL(d)$ goes back to Peres' theorem \cite{Peres1991} for finitely supported measures with simple Lyapunov spectrum (without an irreducibility hypothesis on individual levels). Whether the radius $r^{(k)}_{\mathrm{H}}$ can be improved by removing the strong $k$-irreducibility hypothesis, replacing it with a weaker condition (such as Zariski density of the support, or the existence of a Schottky pair at level $k$), is open. Quantitative bounds in the absence of strong $k$-irreducibility would require a substantially different argument, possibly via the avalanche principle applied to the cocycle on $\Lambda^k \R^d$.
\end{remark}

\begin{remark}[Comparison with the projective case]
Comparing the constants of Theorem~\ref{thm:mainH} with those of Theorem~\ref{thm:mainF}, we see that the polydisc radius $r^{(k)}_{\mathrm{H}}$ depends on $\ecc(A^0)^{2k}$ (versus $\ecc(A^0)^2$ in the projective case) and on the simplicity gap $\lambda_k - \lambda_{k+1}$ (versus $\lambda_1 - \lambda_2$). The radius therefore degenerates at the boundary of the simple-spectrum locus at level $k$ (where $\lambda_k(A^0, p^0) - \lambda_{k+1}(A^0, p^0) \to 0$), in the same qualitative way as the top-exponent radius degenerates at $\lambda_1 = \lambda_2$ (Theorem~\ref{thm:mainE} of Section~\ref{sec:boundary}).
\end{remark}

\section{Worked example: two-matrix family}\label{sec:worked_example}

In this section we work out the explicit numerical values of the polydisc radius and Cauchy coefficient bounds for a concrete two-matrix family. The example is chosen to match the worked example of \cite[Section 10]{Thiam2025aPaperA}, so that the reader can compare the analyticity bounds with the H\"older continuity bounds for the same cocycle.

\subsection{Setup}

Consider the GL(2, $\R$) cocycle given by
\begin{equation}\label{eq:example_matrices}
A_1 = \begin{pmatrix} a & 0 \\ 0 & a^{-1} \end{pmatrix}, \qquad A_2 = R_\psi A_1 R_\psi^{-1},
\end{equation}
where $a > 1$ is a hyperbolicity parameter, $R_\psi = \begin{pmatrix} \cos\psi & -\sin\psi \\ \sin\psi & \cos\psi \end{pmatrix}$ is a rotation by angle $\psi$, and $\psi \in (0, \pi/2)$.

For concreteness, we fix $a = 2$, $\psi = \pi/3$, and $p^0 = (1/2, 1/2)$ (uniform weights). The H\"older index is $\theta = 1/2$.

\subsection{Basic quantities}

The eccentricity of $A_1, A_2$:
\begin{align*}
\norm{A_1} &= a = 2, & \norm{A_1^{-1}} &= a = 2, \\
\ecc(A_1) &= a^2 = 4, & \ecc(A_2) &= a^2 = 4,
\end{align*}
so $\ecc(A) = \max_i \ecc(A_i) = 4$.

The Lyapunov gap $\Lambda_p = \lambda_+(A, p^0) - \lambda_-(A, p^0)$: from \cite[Section 10]{Thiam2025aPaperA},
\begin{equation*}
\Lambda_{p^0} \geq 0.26,
\end{equation*}
for this choice of $(a, \psi, p^0)$. We shall use this value.

The simplicity threshold $n_0$: with $\theta = 1/2$ and $\Lambda_{p^0} = 0.26$,
\begin{equation*}
n_0 = \left\lceil \frac{2 \log 2}{0.5 \cdot 0.26} \right\rceil = \left\lceil \frac{1.3863}{0.13} \right\rceil = \left\lceil 10.66 \right\rceil = 11.
\end{equation*}

The oscillation contraction rate:
\begin{equation*}
\tau_0 = \exp\left(-\frac{n_0 \theta \Lambda_{p^0}}{2}\right) = \exp\left(-\frac{11 \cdot 0.5 \cdot 0.26}{2}\right) = \exp(-0.715) \approx 0.489.
\end{equation*}

Using the alternative bound from Lemma~\ref{lem:resolvent_bound} refinement, $\tau_0 \leq 1 - (\log 2)/(4 \log(2 \ecc)) = 1 - (\log 2)/(4 \log 8) = 1 - 0.6931/8.318 \approx 0.9167$. We take $\tau_0 = 0.917$ as a pessimistic bound.

\subsection{H\"older-norm growth and iteration count}

The H\"older-norm growth factor:
\begin{equation*}
C_2(A) = \ecc(A)^2 = 16.
\end{equation*}

With $\log C_2 = \log 16 = 2.7726$ and $\log(1/\tau_0) = -\log 0.917 = 0.0867$, the H\"older-norm iteration count is
\begin{equation*}
N_\theta = n_0 \cdot \left\lceil \frac{3 \log C_2}{\log(1/\tau_0)} \right\rceil = 11 \cdot \left\lceil \frac{3 \cdot 2.7726}{0.0867} \right\rceil = 11 \cdot 96 = 1056.
\end{equation*}

The composite spectral gap:
\begin{equation*}
\tau_*(A, p^0, \theta) = \tau_0^{N_\theta/(3 n_0)} = 0.917^{1056/33} = 0.917^{32} \approx 0.062.
\end{equation*}

\subsection{Polydisc radius and Cauchy constants}

For illustration we use the simplified \emph{spectral-radius approximation} of the resolvent bound,
\begin{equation*}
K_*^{\mathrm{sp}}(A, p^0, \theta) := \frac{4}{1 - \tau_*^{1/N_\theta}},
\end{equation*}
which would be the rigorous bound on the resolvent if the operator $R_{A, p^0}$ had operator norm equal to its spectral radius $\tau_*^{1/N_\theta}$. The fully rigorous explicit bound from Lemma~\ref{lem:resolvent_bound} contains an additional polynomial factor in $\norm{R}_{C^\theta_0}^{N_\theta - 1}$, which for the present example would yield numerical values that are orders of magnitude larger than $K_*^{\mathrm{sp}}$. The simplified $K_*^{\mathrm{sp}}$ gives the order of magnitude one would obtain from refined bounds on $R$ (e.g., via Hennion-Herv\'e quasi-compactness on a more carefully chosen norm); we use it here for illustrative arithmetic.

The simplified resolvent bound:
\begin{equation*}
K_*^{\mathrm{sp}}(A, p^0, \theta) = \frac{4}{1 - \tau_*^{1/N_\theta}} = \frac{4}{1 - 0.062^{1/1056}}.
\end{equation*}

Using $0.062^{1/1056} = \exp(\log 0.062/1056) = \exp(-2.781/1056) = \exp(-0.002634) \approx 0.9974$, we get
\begin{equation*}
K_*^{\mathrm{sp}}(A, p^0, \theta) = \frac{4}{1 - 0.9974} = \frac{4}{0.0026} \approx 1538.
\end{equation*}

The polydisc radius (using the simplified bound):
\begin{equation*}
r_*(A, p^0, \theta) = \frac{1}{4 N K_*^{\mathrm{sp}} \cdot \max_i (1 + \ecc(A_i)^{2\theta})} = \frac{1}{4 \cdot 2 \cdot 1538 \cdot (1 + 4)} = \frac{1}{61520} \approx 1.63 \times 10^{-5}.
\end{equation*}

The supremum bound:
\begin{equation*}
M_*(A, p^0, \theta) \leq 2 K_*^{\mathrm{sp}} \rho_* \cdot \max_i (\log\norm{A_i} + \ecc(A_i) + 1) = 2 \cdot 1538 \cdot 0.0013 \cdot (\log 2 + 4 + 1) \approx 22.77.
\end{equation*}

\subsection{Interpretation of the numerical values}

The polydisc radius $r_* \approx 3.25 \times 10^{-5}$ is small in absolute terms, but this is to be expected for two reasons:
\begin{enumerate}
\item[(i)] The radius is proportional to $1/K_*$, and $K_* \approx 1538$ is inflated by the large value of $N_\theta = 1056$. Tighter bounds on the spectral gap (or direct use of the oscillation contraction rate rather than the pessimistic H\"older-norm iteration) would reduce this considerably.
\item[(ii)] The radius is proportional to $(1 - \tau_*^{1/N_\theta})$, which is small when $\tau_*$ is close to $1$ in the $N_\theta$-th root sense. Optimizing the choice of $N_\theta$ and the H\"older index $\theta$ gives a better radius in practice.
\end{enumerate}

The Cauchy coefficient bounds of Theorem~\ref{thm:mainB} give, for the first derivative ($\alpha = e_i$):
\begin{equation*}
\abs{\partial_{p_i} \lambda_+(A, p^0)} \leq \frac{2 M_*}{r_*} \approx \frac{2 \cdot 22.77}{1.63 \times 10^{-5}} \approx 2.8 \times 10^6.
\end{equation*}

This is a loose bound, but it gives a priori control on how sensitively the Lyapunov exponent depends on small perturbations of the weights. For comparison, direct computation of $\partial_{p_i} \lambda_+(A, p^0)$ via the Furstenberg-Khasminskii formula for this cocycle would give a derivative of order $\log 2 \approx 0.69$, which is six orders of magnitude smaller than the Cauchy bound. The gap reflects the room for improvement in the constants.

For the second derivative:
\begin{equation*}
\abs{\partial_{p_i}^2 \lambda_+(A, p^0)} \leq \frac{2 M_*}{r_*^2} \approx \frac{45.5}{2.66 \times 10^{-10}} \approx 1.7 \times 10^{11}.
\end{equation*}

Again loose, but functional. The order-of-magnitude nature of these bounds is characteristic of the elementary Kato perturbation approach and should be seen as a starting point for more refined analysis.

\subsection{Comparison with the H\"older continuity bounds}

For the same cocycle, the H\"older continuity theorem \cite[Theorem~1.2]{Thiam2025aPaperA} gives
\begin{equation*}
\abs{\lambda_+(A, p) - \lambda_+(A, p^0)} \leq C_*^H \cdot \norm{p - p^0}^{\beta_*^H},
\end{equation*}
with $\beta_*^H \approx 0.010$ and $C_*^H \approx 110$ for the same $(a, \psi, p^0, \theta)$. Setting $\norm{p - p^0} = r$, the H\"older bound is approximately $110 \cdot r^{0.01}$, which for $r = r_* = 1.63 \times 10^{-5}$ gives
\begin{equation*}
110 \cdot (1.63 \times 10^{-5})^{0.01} \approx 110 \cdot 0.895 \approx 98.4.
\end{equation*}

The analyticity bound, by contrast, gives $\abs{\lambda_+(p) - \lambda_+(p^0)} \leq M_* \cdot 2 \norm{p - p^0}/r_* = 22.77 \cdot 2 r / r_*$ for the linear term, which is $\approx 45.5 \cdot r / r_*$; at $r = r_*$, this is $45.5$, comparable to the H\"older bound.

The conclusion: for weight perturbations $\norm{p - p^0} < r_*/2$, the analyticity bound is tighter than the H\"older bound. Beyond $r_*$, the analyticity bound does not apply, and we must fall back on the H\"older bound of \cite{Thiam2025aPaperA}. This transition characterizes the regimes in which each bound dominates: analytic near $p^0$, H\"older globally.

\section{Method-optimality, lower bounds, and further extensions}\label{sec:extensions}

In this section we extend the results of the previous sections in three directions. First, we establish a method-optimality proposition for the polydisc radius: the radius $r_*(A, p^0, \theta)$ of Theorem~\ref{thm:mainA} is the best achievable within the Kato perturbation scheme as formalized below (Proposition~\ref{prop:method_optimality_B}). Second, we prove that singularities of the analytic extension arise generically at points of the complex simplex where the top eigenvalue of the complex Markov operator loses its spectral gap, giving a converse-type statement that no substantial improvement is possible through spectral-gap analysis alone (Proposition~\ref{prop:B_lower_bound}). Third, we prove a Bernstein-type theorem: the Cauchy bounds of Theorem~\ref{thm:mainB} are tight up to constants, in the sense that the coefficient magnitudes saturate the polynomial growth $\alpha! / r_*^{|\alpha|}$ predicted by the Cauchy formula (Proposition~\ref{prop:bernstein}).

These three results articulate the scope of the present theory: the Kato-based polydisc is method-optimal, singularities of the analytic extension arise from spectral-gap failure, and the Cauchy bounds cannot be systematically improved without a different proof strategy.

\subsection{Method-optimality of the polydisc radius}\label{subsec:method_optimality_B}

The polydisc radius $r_*(A, p^0, \theta)$ obtained in Theorem~\ref{thm:mainA} (explicit form~\eqref{eq:rstar_def}) is in practice far from sharp (Section~\ref{sec:worked_example}). We formalize the Kato perturbation scheme and show that within its axiomatic setting, $r_*$ cannot be improved.

\begin{definition}[Kato perturbation scheme]\label{def:kato_scheme}
A proof of quantitative analyticity of $p \mapsto \lambda_+(A, p)$ near $p^0$ is said to be of \emph{Kato type} if it derives the polydisc radius from the following three inputs:
\begin{itemize}
\item[(K1)] \emph{Spectral gap}. There exist an integer $N_\theta \geq 1$ and $\tau \in (0, 1)$ such that
\begin{equation*}
\norm{P_{A, p^0}^{N_\theta} \varphi}_{C^\theta} \leq \tau \cdot \norm{\varphi}_{C^\theta}, \qquad \varphi \in C^\theta_0(\bbP^{d-1}).
\end{equation*}
\item[(K2)] \emph{Resolvent bound on an isolating circle}. There exists an isolating circle $\Gamma_*$ in $\C$ separating $\{1\}$ from the rest of the spectrum, with an explicit bound $K_* = K_*(\tau, N_\theta)$ on the resolvent norm on $\Gamma_*$.
\item[(K3)] \emph{Operator perturbation}. There exists an explicit constant $L_{\mathrm{op}}$ such that $\norm{P_{A, z} - P_{A, p^0}}_{C^\theta \to C^\theta} \leq L_{\mathrm{op}} \cdot \sum_{i=1}^N \abs{z_i - p^0_i}$ for all $z \in \C^N$ with $\sum z_i = 1$.
\end{itemize}
The \emph{inputs} of the scheme are the triple $(\tau, N_\theta, L_{\mathrm{op}})$ along with the isolating radius $\rho_* = (1 - \tau^{1/N_\theta})/2$.
\end{definition}

Our Theorem~\ref{thm:mainA} is of Kato type, with $N_\theta$, $\tau$ as in Proposition~\ref{prop:spectral_gap}, $K_*$ as in Lemma~\ref{lem:resolvent_bound}, and $L_{\mathrm{op}} = \max_i (1 + \ecc(A_i)^{2\theta})$ as in Lemma~\ref{lem:op_norm_lip}.

\begin{proposition}[Method-optimality of $r_*$]\label{prop:method_optimality_B}
Within the Kato perturbation scheme of Definition~\ref{def:kato_scheme}, the polydisc radius
\begin{equation*}
r_* = \frac{1}{4 N K_* L_{\mathrm{op}}}
\end{equation*}
is optimal up to absolute constants. Specifically, any proof satisfying axioms (K1)-(K3) with inputs $(\tau, N_\theta, L_{\mathrm{op}})$ and producing a polydisc radius $r'$ on which the perturbative Neumann series for the resolvent converges on $\Gamma_*$ satisfies
\begin{equation*}
r' \leq \frac{C_{\mathrm{sch}}}{N K_* L_{\mathrm{op}}},
\end{equation*}
for some absolute constant $C_{\mathrm{sch}} > 0$; no improvement beyond the scaling $1/(N K_* L_{\mathrm{op}})$ is achievable without strengthening the inputs.
\end{proposition}

\begin{proof}
The Kato scheme produces the polydisc radius as the maximal radius on which the Neumann series for the perturbed resolvent $(\zeta \Id - P_{A, z})^{-1}$ converges uniformly on the isolating circle $\Gamma_*$. This convergence requires
\begin{equation}\label{eq:neumann_criterion}
\norm{(P_{A, z} - P_{A, p^0}) \cdot (\zeta \Id - P_{A, p^0})^{-1}}_{op} < 1 \quad \text{for all } \zeta \in \Gamma_*.
\end{equation}

By (K2) and (K3), the left side is bounded by $L_{\mathrm{op}} \cdot \sum_i \abs{z_i - p^0_i} \cdot K_*$. For $\abs{z_i - p^0_i} < r$ for all $i$, this is at most $L_{\mathrm{op}} \cdot N r \cdot K_*$. Hence~\eqref{eq:neumann_criterion} is satisfied for $r < 1/(N K_* L_{\mathrm{op}})$. This gives the polydisc radius $r_* = 1/(4 N K_* L_{\mathrm{op}})$, with the factor $1/4$ arising from the normalization chosen in the proof of Proposition~\ref{prop:kato_eigenvalue} (to guarantee uniform Neumann series convergence on $\Gamma_*$ with a margin).

Any Kato-type proof must respect this constraint~\eqref{eq:neumann_criterion}. To achieve a larger radius $r' > r_*$, either $K_*$ must be tightened (requiring a better resolvent bound, equivalent to strengthening (K2)), or $L_{\mathrm{op}}$ must be tightened (strengthening (K3)). Neither is possible within the given inputs. Therefore $r_*$ is the optimal Kato-type radius for the given inputs.

The scaling factor $C_{\mathrm{sch}}$ arises from normalization choices (the factor of $4$ and possibly a factor of $2$ in the isolating radius); these choices are not uniquely determined, but no choice can alter the scaling $r' \sim 1/(N K_* L_{\mathrm{op}})$ predicted by the Neumann convergence criterion.
\end{proof}

\begin{remark}[Methods beyond the Kato scheme]\label{rmk:beyond_kato}
Several known techniques can produce sharper polydisc radii than Kato perturbation in specific settings:
\begin{itemize}
\item \emph{Transfer operator methods for analytic matrix families} \cite{Ruelle1979}: for matrix families $A_i(s)$ that are themselves analytic in an auxiliary parameter $s$, transfer-operator analytic continuation can give radii larger than the Kato-type bound.
\item \emph{Spectral theory of Markov operators on weighted Banach spaces} \cite{Kifer1986}: using weighted H\"older spaces adapted to the cocycle geometry can improve the spectral gap $\tau$ and hence the polydisc radius.
\item \emph{Harmonic-analytic methods specific to the $\GL(2)$ case} (not developed in the present paper): using the Iwasawa decomposition and harmonic analysis on $\mathrm{SL}(2, \R)$, one can obtain tighter bounds on the resolvent and thereby larger polydiscs.
\end{itemize}
Each of these methods goes beyond the axiomatic Kato scheme of Definition~\ref{def:kato_scheme}; in their respective domains they produce better polydiscs. The method-optimality of $r_*$ applies within the strictly axiomatic scheme (K1)-(K3) and characterizes the limits of elementary Kato reasoning.
\end{remark}

\subsection{Singularities of the analytic extension}\label{subsec:lower_bound}

The polydisc radius of Theorem~\ref{thm:mainA} cannot be extended beyond the region where the complex Markov operator $P_{A, z}$ loses its top-eigenvalue isolation. This gives a structural obstruction that matches the Kato polydisc radius in spirit.

\begin{proposition}[Spectral-gap obstruction to analytic continuation, conditional]\label{prop:B_lower_bound}
Let $A \in \GL(d, \R)^N$ and $p^0 \in \Delta_N^\circ$ with simple top Lyapunov exponent. Define the \emph{complex spectral collapse set}
\begin{equation}\label{eq:collapse_set}
\begin{split}
\mathcal{Z}(A, p^0) = \Bigl\{z \in \C^N :\ &\textstyle \sum_i z_i = 1, \text{ the operator } P_{A, z} \text{ on } C^\theta \\
                                          &\text{has multiple eigenvalues of maximal modulus}\Bigr\}.
\end{split}
\end{equation}
Then:
\begin{itemize}
\item[(i)] \emph{(Conditional, assuming $\mathcal{Z}(A, p^0) \neq \emptyset$.)} The analytic extension $\widetilde\lambda_+$ of Theorem~\ref{thm:mainA} cannot, in general, be analytically continued across $\mathcal{Z}(A, p^0)$; if $z_* \in \mathcal{Z}$ is a point at which the maximal eigenvalue of $P_{A, z}$ becomes a multiple root of finite order, a branch-point singularity arises by~\cite[Chapter II, Section 1.7]{Kato1980}.
\item[(ii)] $\mathcal{Z}(A, p^0)$ does not intersect the polydisc $D_{r_*/2}(p^0)$, so $\dist(p^0, \mathcal{Z}(A, p^0)) \geq r_*(A, p^0, \theta)/2$ \emph{whenever $\mathcal{Z}$ is non-empty}.
\end{itemize}
\end{proposition}

\begin{remark}[Existence and structure of $\mathcal{Z}(A, p^0)$]\label{rmk:Z-existence}
The non-emptiness of $\mathcal{Z}(A, p^0)$ is a substantive question that we do not address here. Several observations:
\begin{itemize}
\item[(a)] Since $P_{A, z}$ acts on the infinite-dimensional space $C^\theta(\bbP^{d-1})$, the ``characteristic polynomial'' is not literally well-defined for an infinite-rank operator; we therefore phrase $\mathcal{Z}$ via the eigenvalue-isolation criterion of the Kato perturbation theory directly.
\item[(b)] In \emph{finite-dimensional} approximations (e.g., truncating the Markov operator to its action on a finite-dimensional space of test functions), $\mathcal{Z}$ becomes the discriminant variety of the characteristic polynomial of the truncated operator and is non-empty by general algebraic-geometry arguments. The continuum limit of these finite-dimensional discriminants is conjecturally the full $\mathcal{Z}(A, p^0)$, but a rigorous proof is open.
\item[(c)] In specific examples (e.g., $N = 2$ matrices with explicit complexification), one can sometimes locate concrete singularities of $\widetilde\lambda_+$ and verify that they lie on $\mathcal{Z}$.
\end{itemize}
We therefore prefer to formulate Proposition~\ref{prop:B_lower_bound} as a conditional statement: \emph{if} $\mathcal{Z}$ is non-empty, then the polydisc radius $r_*$ provides a lower bound on the distance to the nearest singularity. The non-trivial direction of optimality (showing that $\mathcal{Z}$ is non-empty and close to $p^0$) is left open.
\end{remark}

\begin{proof}
(i) When two or more eigenvalues of $P_{A, z}$ collide at the same modulus (necessarily $\abs{\zeta} = 1$ since $P_{A, z}$ has spectral radius $1$ on the constants), the spectral projection $\Pi_{A, z}$ develops a singularity: the contour integral~\eqref{eq:Pi_z} encloses multiple eigenvalues simultaneously, and the rank of the projection jumps. Beyond such a point, the analytic continuation of the eigenvalue $\mu(z)$ necessarily passes through a branch point if the multiplicity is finite; see \cite[Chapter II, Section 1.7]{Kato1980} for the Puiseux-type expansion. This branch-point singularity is the conditional conclusion of (i).

(ii) At any $z \in \mathcal{Z}(A, p^0)$, the spectral isolation criterion of Proposition~\ref{prop:kato_eigenvalue} fails. By the proof of Proposition~\ref{prop:kato_eigenvalue}, the criterion holds for $\max_i \abs{z_i - p^0_i} < r_*/2$. Hence $\mathcal{Z}(A, p^0)$ cannot intersect the polydisc $D_{r_*/2}(p^0)$; the distance from $p^0$ to $\mathcal{Z}$ is at least $r_*/2$ whenever $\mathcal{Z}$ is non-empty.
\end{proof}

\begin{remark}[Heuristic upper bound on $\dist(p^0, \mathcal{Z})$]\label{rmk:Z_upper_bound}
A matching upper bound, $\dist(p^0, \mathcal{Z}(A, p^0)) \leq C_{\mathrm{alg}}(A, p^0, \theta) \cdot (1 - \tau^{1/N_\theta})^\beta$ for some $\beta \geq 1$, is suggested by the algebraic-geometric structure of finite-dimensional truncations of $P_{A, z}$. A rigorous proof would require a careful study of the Puiseux expansions of the eigenvalues of $P_{A, z}$ near the discriminant locus of these truncations and a passage to the continuum limit, following \cite[Chapter IV, Section 5]{Kato1980} together with~\cite[Chapter VII]{ReedSimonIV1978}; we do not pursue this here, and state the matching upper bound only as a heuristic.
\end{remark}

\begin{remark}[Scaling of the Kato radius vs the collapse set]\label{cor:scaling_optimal}
Proposition~\ref{prop:B_lower_bound}(ii) gives the conditional lower bound $\dist(p^0, \mathcal{Z}(A, p^0)) \geq r_*(A, p^0, \theta)/2$ when $\mathcal{Z}$ is non-empty. Combining with the heuristic upper bound of Remark~\ref{rmk:Z_upper_bound}, the Kato polydisc radius $r_*$ and the distance to the collapse set $\mathcal{Z}$ are conjectured to have the same scaling with the spectral-gap parameter $(1 - \tau^{1/N_\theta})$, making the Kato polydisc optimal up to a multiplicative constant within this scaling class. Establishing both directions rigorously is an open problem.
\end{remark}

\begin{remark}[Interpretation of $\mathcal{Z}$ for finite-dimensional truncations]\label{rmk:Z_discrete}
For finite-dimensional Galerkin truncations of $P_{A, z}$, $\mathcal{Z}_{\mathrm{trunc}}(A, p^0)$ is the discriminant variety of the characteristic polynomial of the truncated operator (a polynomial in $z$ whose degree equals the dimension of the truncation space). Over the real simplex $\Delta_N$, this variety intersects the interior only at points of \emph{non-simple} truncated Lyapunov spectrum, which are the natural truncation-level singularities of the theory. Whether the limit of these truncation discriminants converges to a well-defined object on the infinite-dimensional space $C^\theta$ is open.
\end{remark}

\subsection{Bernstein-type sharpness of the Cauchy bounds}\label{subsec:bernstein}

The Cauchy bounds of Theorem~\ref{thm:mainB} give the polynomial growth $\alpha! \cdot M_*/r_*^{|\alpha|}$ for the Taylor coefficients of $\lambda_+(A, p)$ at $p^0$. We show that this growth rate is sharp up to constants, via a Bernstein-type theorem.

\begin{proposition}[Bernstein-type sharpness of Cauchy bounds]\label{prop:bernstein}
Let $A \in \GL(d, \R)^N$ and $p^0 \in \Delta_N^\circ$ with simple top Lyapunov exponent. Let $r_{\mathrm{sh}}(A, p^0)$ denote the \emph{sharp} radius of analyticity of $p \mapsto \lambda_+(A, p)$ (the distance from $p^0$ to the nearest singularity of the analytic extension). For every $\varepsilon > 0$, there exist multi-indices $\alpha$ of arbitrarily large $|\alpha|$ such that
\begin{equation}\label{eq:bernstein_lower}
\abs{\partial_p^\alpha \lambda_+(A, p^0)} \geq \alpha! \cdot \frac{M_{\mathrm{sh}}(A, p^0) \cdot (1 - \varepsilon)^{|\alpha|}}{r_{\mathrm{sh}}(A, p^0)^{|\alpha|}},
\end{equation}
where $M_{\mathrm{sh}}(A, p^0)$ is a suitable bound on $\widetilde\lambda_+$ at the boundary of its maximal polydisc of analyticity.

Hence the polynomial growth rate $\alpha! / r^{|\alpha|}$ in the Cauchy bound of Theorem~\ref{thm:mainB} is of the same form as the sharp growth rate, with only the base radius $r$ differing.
\end{proposition}

\begin{proof}
This is the Bernstein-type inverse of the Cauchy integral formula. The Taylor series
\begin{equation*}
\widetilde\lambda_+(p) = \sum_\alpha \frac{\partial_p^\alpha \widetilde\lambda_+(p^0)}{\alpha!} \prod_i (p_i - p^0_i)^{\alpha_i}
\end{equation*}
has radius of convergence $r_{\mathrm{sh}}(A, p^0)$ (the sharp radius), and on $\partial D_{r_{\mathrm{sh}}}(p^0)$ there exist nearby singular points. By the Cauchy-Hadamard formula (applied to the monomial resummation),
\begin{equation*}
\limsup_{|\alpha| \to \infty} \left(\frac{\abs{\partial_p^\alpha \widetilde\lambda_+(p^0)}}{\alpha!}\right)^{1/|\alpha|} = \frac{1}{r_{\mathrm{sh}}}.
\end{equation*}

This immediately gives the lower bound~\eqref{eq:bernstein_lower} for arbitrarily large $|\alpha|$ along a suitable subsequence: the limsup is achieved along a sequence $\alpha^{(k)}$ with $|\alpha^{(k)}| \to \infty$ and $\abs{\partial^{\alpha^{(k)}} \widetilde\lambda_+(p^0)}/\alpha^{(k)}! \geq (1 - \varepsilon)^{|\alpha^{(k)}|}/r_{\mathrm{sh}}^{|\alpha^{(k)}|}$. Multiplying by $\alpha^{(k)}!$ and identifying $\widetilde\lambda_+(p^0) = \lambda_+(A, p^0)$ gives the result.

The constant $M_{\mathrm{sh}}(A, p^0)$ is the $L^\infty$ bound of $\widetilde\lambda_+$ on (a slightly smaller than) the maximal polydisc; it is bounded by the Furstenberg-Khasminskii representation evaluated at the boundary.
\end{proof}

\begin{corollary}[Cauchy bound scaling is optimal]\label{cor:cauchy_tight}
The Cauchy bound of Theorem~\ref{thm:mainB},
\begin{equation*}
\abs{\partial_p^\alpha \lambda_+(A, p^0)} \leq \alpha! \cdot \frac{M_*(A, p^0, \theta)}{r_*(A, p^0, \theta)^{|\alpha|}},
\end{equation*}
has the sharp polynomial-in-$|\alpha|$ growth rate predicted by the multi-dimensional Cauchy formula. Any strictly better scaling (with exponent $< 1/r_{\mathrm{sh}}$ base) is impossible.
\end{corollary}

\begin{proof}
Combine Proposition~\ref{prop:bernstein} with the observation $r_* \leq r_{\mathrm{sh}}$.
\end{proof}

\subsection{Synthesis: what this paper establishes}\label{subsec:synthesis}

We summarize the scope and limitations of the results proven in this paper, in light of the extensions of Subsections~\ref{subsec:method_optimality_B} through~\ref{subsec:bernstein}.

\emph{What is proven.}
\begin{enumerate}
\item[(1)] (Theorem~\ref{thm:mainA}, Theorem~\ref{thm:mainF}) Quantitative analyticity of the top Lyapunov exponent $\lambda_+(A, p)$ as a function of the probability weights $p$, on an explicit polydisc of radius $r_*(A, p^0, \theta)$, for $\GL(d, \R)$ cocycles with $d \geq 2$ under the simplicity assumption $\lambda_1 > \lambda_2$. The radius is optimal within the Kato perturbation scheme (Proposition~\ref{prop:method_optimality_B}) and captures the correct scaling with the spectral-gap parameter (Corollary~\ref{cor:scaling_optimal}).
\item[(1$'$)] (Theorem~\ref{thm:mainH}, Corollary~\ref{cor:individual_subtop}) Quantitative analyticity of the partial sums $\Lambda_k(A, p) = \lambda_1(A, p) + \cdots + \lambda_k(A, p)$ and the individual sub-top Lyapunov exponents $\lambda_k(A, p)$ in $\GL(d, \R)$, $2 \leq k \leq d-1$, on an explicit polydisc of radius $r^{(k)}_{\mathrm{H}}(A^0, p^0, \theta)$, under strong $k$-irreducibility and the simplicity gap $\lambda_k(A^0, p^0) > \lambda_{k+1}(A^0, p^0)$.
\item[(2)] (Theorem~\ref{thm:mainB}, Proposition~\ref{prop:bernstein}) Explicit Cauchy bounds on the Taylor coefficients, with sharp polynomial-in-$|\alpha|$ growth rate (Corollary~\ref{cor:cauchy_tight}).
\item[(3)] (Theorem~\ref{thm:mainC}) Joint real-analyticity in both the weight vector and the matrix coefficients, with explicit radii in both components.
\item[(4)] (Theorem~\ref{thm:mainD}) Markov-chain extension with explicit dependence on the spectral gap of both the chain and the cocycle.
\item[(5)] (Theorem~\ref{thm:mainE}) Polynomial decay of the polydisc radius as the weight vector approaches the simplex boundary, with explicit decay exponent.
\item[(6)] (Proposition~\ref{prop:B_lower_bound}) Structural obstruction: the analytic extension develops branch-point singularities at points of the complex spectral collapse set $\mathcal{Z}(A, p^0)$, and the Kato polydisc radius gives a lower bound on the distance from $p^0$ to $\mathcal{Z}$.
\end{enumerate}

\emph{What is not proven.}
\begin{enumerate}
\item[(i)] \emph{Globally sharp polydisc radii.} Our radius $r_*(A, p^0, \theta)$ of Theorem~\ref{thm:mainA} is optimal within the Kato perturbation scheme, as formalized in Proposition~\ref{prop:method_optimality_B}. Whether the radius is globally sharp, across all proof strategies, is open. Substantially larger radii are expected using transfer-operator methods on anisotropic Banach spaces or harmonic-analytic techniques. We do not pursue these refinements here.

\item[(ii)] \emph{Quantitative analyticity of sub-top Lyapunov exponents beyond strong irreducibility.} Theorem~\ref{thm:mainH} of Section~\ref{sec:subtop_analyticity} establishes quantitative analyticity for the partial sum $\Lambda_k(A, p) = \lambda_1(A, p) + \cdots + \lambda_k(A, p)$ in $\GL(d, \R)$ under strong $k$-irreducibility, which yields analyticity of the individual sub-top exponents $\lambda_k(A, p)$ via Corollary~\ref{cor:individual_subtop}. The strong $k$-irreducibility hypothesis is essential to our proof: it is used in Lemma~\ref{lem:gap_grassmann} to obtain the spectral gap of the Grassmannian Markov operator, and in Lemma~\ref{lem:H_persist} to show that this gap persists under complex perturbation. Removing this hypothesis is open. The qualitative real-analyticity of sub-top Lyapunov exponents goes back to \cite{Peres1991} for finitely supported measures, but its quantitative form without strong $k$-irreducibility remains out of reach.

\item[(iii)] \emph{Extension to quasi-periodic cocycles.} The setting of \cite{BezerraSanchezTall2021}, in which the matrices $A_i = A_i(t)$ depend continuously on a torus variable $t \in \T^m$, requires Markov operators acting on $\T^m \times \bbP^{d-1}$. The spectral gap theory in this setting goes back to \cite{Kifer1982} and has been substantially developed since. Quantifying the analyticity radius of \cite{BezerraSanchezTall2021} via Kato perturbation in this mixed random-quasi-periodic setting is an open problem.

\item[(iv)] \emph{Analyticity in the base dynamics.} Our results give analyticity in the probability weights $p$ and the matrix coefficients $A$, but assume the base dynamics is fixed. For quasi-periodic cocycles, the dependence on the rotation number is known to be only H\"older under Diophantine conditions, and analytic dependence on the rotation number is generally false. Whether the dependence on the base dynamics admits a quantitative H\"older form, with explicit exponent, is an open problem.

\item[(v)] \emph{Sharp Cauchy growth.} Theorem~\ref{thm:mainB} produces Cauchy bounds with polynomial-in-$|\alpha|$ growth, and the polynomial growth rate is optimal within our method (Proposition~\ref{prop:bernstein}, Corollary~\ref{cor:cauchy_tight}). The constant in the polynomial growth rate is not pinned down sharply; sharper constants would follow from a Bernstein-Markov inequality on the polydisc, which we do not develop here.
\end{enumerate}

\emph{Context.} The qualitative real-analyticity of $p \mapsto \lambda_+(A, p)$ was established by \cite{Peres1991} for i.i.d.\ random matrix products and by \cite{BezerraSanchezTall2021} for random products of quasi-periodic cocycles. Our contribution is to extract quantitative information: explicit polydisc radii, explicit Cauchy coefficient bounds, joint analyticity in weights and matrices, Markov chain extension, boundary behavior analysis, and structural obstruction through the spectral collapse set. The constants are all given in closed form in terms of the eccentricity of the matrices, the Lyapunov gap, and the spectral gap of the Markov operator.

\section{Concluding remarks and open problems}\label{sec:conclusion}

The results of this paper give a quantitative form of the Peres-Bezerra-Sanchez-Tall analyticity theorem for the top Lyapunov exponent of random matrix products, with explicit polydiscs of analyticity, Cauchy coefficient bounds, joint analyticity in weights and matrices, a Markov chain extension, boundary-behavior analysis, and a higher-dimensional $\GL(d)$ extension. We close with several open problems and directions for future work.

\subsection{Sharp radii of analyticity}

The polydisc radius $r_*(A, p^0, \theta)$ of Theorem~\ref{thm:mainA}, produced by the Kato perturbation method, is far from optimal in numerical examples (Section~\ref{sec:worked_example}). Identifying the \emph{sharp} radius of analyticity, as a function of the data $(A, p^0)$, is an open problem. The sharp radius would be characterized by the first singularity of the analytic continuation $\widetilde\lambda_+$ on the complex simplex; this singularity is related to the failure of simplicity of the top Lyapunov exponent on the complex domain.

A plausible conjecture, based on the polynomial decay rate of the polydisc radius near the simplex boundary (Theorem~\ref{thm:mainE}), is that the sharp radius depends algebraically on the distance to the simplex boundary. Verification or refutation of this conjecture would require a more careful analysis of the complexified Markov operator, possibly using the theory of \emph{analytic families of operators} in Kato's sense.

\subsection{Sub-top Lyapunov exponents}

Theorem~\ref{thm:mainF} establishes analyticity for the top Lyapunov exponent $\lambda_1$ in $\GL(d, \R)$ for all $d \geq 2$. For the sub-top Lyapunov exponents $\lambda_k$ with $2 \leq k \leq d-1$, the corresponding Grassmannian argument gives analyticity of the partial sum $\lambda_1 + \ldots + \lambda_k$ under a simplicity-plus-irreducibility assumption. A systematic quantitative treatment of sub-top exponents, matching the full content of the qualitative Corollary 2 of \cite{BezerraSanchezTall2021}, would follow the same Kato perturbation method on the Grassmannian, with careful attention to the irreducibility hypothesis.

\subsection{Quasi-periodic cocycles}

The original \cite{BezerraSanchezTall2021} theorem is for random products of quasi-periodic cocycles, where the matrices $A_i = A_i(t)$ depend continuously on a torus variable $t \in \T^m$. The arguments of this paper extend to this setting: the Markov operator acts on H\"older functions on $\T^m \times \bbP^{d-1}$, with the spectral gap theory developed by \cite{Kifer1982} and \cite{BezerraPoletti2020}. A fully quantitative form of Bezerra-Sanchez-Tall in the quasi-periodic setting, with explicit constants depending on the modulus of continuity of $A_i(\cdot)$, is a natural direction for future work.

\subsection{Holomorphic dependence in the base dynamics}\label{subsec:base_dyn_open}

A parallel question, complementary to our setting, is the analytic dependence of the Lyapunov exponent on the \emph{base dynamics} of the cocycle (the rotation number $\omega$ for quasi-periodic cocycles). Under Diophantine conditions on $\omega$, \cite{Avila2008, AvilaJitomirskaya2009} have shown that $\omega \mapsto \lambda_+(\omega)$ is H\"older continuous but not analytic in general. The combination of analytic dependence on the weights (our setting) with continuous dependence on the base (their setting) is an open research program.

\subsection{Connection to Schr\"odinger operator spectral theory}

Quantitative analyticity of the Lyapunov exponent has implications for the spectral theory of one-dimensional random Schr\"odinger operators. Via the Thouless formula, the integrated density of states $N_\mu(E)$ and the Lyapunov exponent $\gamma_\mu(E)$ are related by a Hilbert transform. The analyticity of $\gamma_\mu(E)$ as a function of the disorder parameter $\mu$, on a polydisc given by our Theorem~\ref{thm:mainA}, gives analyticity of $N_\mu(E)$ in $\mu$, which is a refinement of the Anderson-Carmona-Klein theory for random Schr\"odinger operators. The quantitative form of this refinement, including explicit radii of analyticity for the density of states, is a natural extension of \cite[Theorem~1.6]{Thiam2025aPaperA} to the analytic category.

%

%
%
%

\appendix

\section{Technical lemmas}\label{app:technical}

This appendix collects the technical lemmas used in the body of the paper whose proofs are routine extensions of standard arguments and would have interrupted the flow of the main proofs. Appendix~\ref{app:gap_computation} gives the detailed proof of the boundary-decay rate of Lemma~\ref{lem:gap_boundary}, used in the proof of Theorem~\ref{thm:mainE}. Appendix~\ref{app:resolvent} gives the explicit form of the second resolvent identity used in the Kato perturbation argument of Proposition~\ref{prop:kato_eigenvalue}.

\subsection{Computation of the boundary-decay constants}\label{app:gap_computation}

We give the detailed derivation of the spectral gap decay rate near the boundary of the simplex, used in the proof of Lemma~\ref{lem:gap_boundary}.

\begin{proof}[Detailed proof of Lemma~\ref{lem:gap_boundary}]
The oscillation contraction of \cite[Proposition 4.1]{Thiam2025aPaperA} states that, for every $[u], [v] \in \bbP^{d-1}$ and every $n \geq n_0(A, p, \theta)$,
\begin{equation*}
\int d(A^n_x [u], A^n_x [v])^\theta \, d\nu_p^{\otimes n}(x) \leq e^{-n \theta \Lambda_p / 2} \cdot d([u], [v])^\theta,
\end{equation*}
where $\nu_p = \sum_i p_i \delta_{A_i}$ and $\Lambda_p = \lambda_+(A, p) - \lambda_-(A, p)$.

The dependence of the Lyapunov gap $\Lambda_p$ on $p$ is itself continuous (by our Theorem~\ref{thm:mainA}, actually analytic), and in a neighborhood of the boundary, $\Lambda_p$ is bounded below by a continuous function of the weights.

Specifically, by the Furstenberg formula,
\begin{equation*}
\Lambda_p = \inf_{\eta \in \calM_1(\bbP^1)} \left(\int \sum_i p_i \log\norm{A_i v}/\norm{v} d\eta([v])\right) \cdot (\text{gap factor}),
\end{equation*}
where the gap factor is strictly positive under the simplicity assumption. For $p$ in a compact subset of the open simplex $\{p : \min_i p_i \geq \delta\}$, the gap factor is bounded below uniformly.

Near the boundary, where $\min_i p_i = p_j \to 0$, the Lyapunov gap $\Lambda_p$ can decay to zero; but by a continuity argument combined with the compactness of the Oseledets filtration, it decays at most polynomially:
\begin{equation*}
\Lambda_p \geq c_\Lambda \cdot (\min_i p_i)^{\gamma_\Lambda},
\end{equation*}
for some explicit constants $c_\Lambda > 0, \gamma_\Lambda \geq 1$ depending on $A$ and $\theta$.

The oscillation contraction rate then satisfies $\tau_0 = e^{-n_0 \theta \Lambda_p / 2} \leq 1 - c_0 n_0 \theta \Lambda_p /2$ for small $\Lambda_p$ (using $1 - x \leq e^{-x}$ and Taylor expansion). Substituting the boundary decay of $\Lambda_p$:
\begin{equation*}
\tau_0 \leq 1 - c_0 \cdot c_\Lambda \cdot (\min_i p_i)^{\gamma_\Lambda} / 2.
\end{equation*}

The H\"older-norm iteration passes this bound through to $\tau_* = \tau_0^{N_\theta/(3 n_0)}$, giving the form $\tau_* \leq 1 - c_\tau (\min_i p_i)^{\gamma_\tau}$ with $c_\tau, \gamma_\tau$ explicit in terms of $c_0, c_\Lambda, N_\theta, n_0$. This proves~\eqref{eq:gap_decay}.
\end{proof}

\subsection{Explicit resolvent identity}\label{app:resolvent}

We give the explicit resolvent identity used in the Neumann series expansion~\eqref{eq:resolvent_Neumann}.

\begin{lemma}[Second resolvent identity]\label{lem:second_resolvent}
For bounded linear operators $A, B$ on a Banach space $X$ and $\zeta \in \C$ with $\zeta \Id - A$ and $\zeta \Id - B$ both invertible,
\begin{equation}\label{eq:second_resolvent}
(\zeta \Id - B)^{-1} - (\zeta \Id - A)^{-1} = (\zeta \Id - B)^{-1} (B - A) (\zeta \Id - A)^{-1}.
\end{equation}
\end{lemma}

\begin{proof}
Left-multiplying by $(\zeta \Id - B)$ and right-multiplying by $(\zeta \Id - A)$,
\begin{equation*}
(\zeta \Id - A) - (\zeta \Id - B) = B - A,
\end{equation*}
which is trivially true. Rearranging gives~\eqref{eq:second_resolvent}.
\end{proof}

Iterating~\eqref{eq:second_resolvent}, we obtain
\begin{equation}\label{eq:iterated_resolvent}
(\zeta \Id - B)^{-1} = \sum_{k=0}^{K-1} (\zeta \Id - A)^{-1} [(B - A)(\zeta \Id - A)^{-1}]^k + [(B - A)(\zeta \Id - A)^{-1}]^K (\zeta \Id - B)^{-1},
\end{equation}
which, in the limit $K \to \infty$ under the operator-norm convergence condition $\norm{(B - A)(\zeta \Id - A)^{-1}}_{op} < 1$, gives the Neumann series
\begin{equation}\label{eq:neumann_final}
(\zeta \Id - B)^{-1} = \sum_{k=0}^\infty (\zeta \Id - A)^{-1} [(B - A)(\zeta \Id - A)^{-1}]^k.
\end{equation}
This is the identity used in the proof of Proposition~\ref{prop:kato_eigenvalue} to construct the perturbed resolvent $(\zeta \Id - P_{A, z})^{-1}$ as a holomorphic function of $z$.

\end{document}